\documentclass[10pt,leqno]{amsart}

\usepackage{packages}

\title[Dual Hoffman Bounds Based on SDP]{%
  Dual Hoffman Bounds for the Stability\\
  and Chromatic Numbers Based on SDP}
\author[Nathan Benedetto Proença]{Nathan Benedetto Proença\textsuperscript{1\P}}
\email{nathan@ime.usp.br}
\thanks{%
  \textsuperscript{\P}%
  This work was partially supported by Conselho Nacional de
  Desenvolvimento Científico e Tecnológico (CNPq).
  The author acknowledges CAPES (PROEX) for partial support of this
  work.}
\author[Marcel K. de Carli Silva]{Marcel K. de Carli Silva\textsuperscript{1*}}
\address{%
  \textsuperscript{1}Universidade de São Paulo, Instituto de Matemática e Estatística%
}
\email{mksilva@ime.usp.br}
\thanks{%
  \textsuperscript{*}%
  This work was partially supported by CNPq
  (Proc.~423833/2018-9, 456792/2014-7, and~477203/2012-4), by grant
  \#2013/03447-6, São Paulo Research Foundation (FAPESP), and by CAPES
  (PROEX)}
\author[Gabriel Coutinho]{Gabriel Coutinho\textsuperscript{2}}
\address{%
  \textsuperscript{2}Dep.\ of Computer Science, Federal University of Minas Gerais%
}
\email{gabriel@dcc.ufmg.br}

\date{April 10, 2020}

\begin{document}

\begin{abstract}
  The notion of duality is a key element in understanding the
  interplay between the stability and chromatic numbers of a graph.
  This notion is a central aspect in the celebrated theory of perfect
  graphs, and is further and deeply developed in the context of the
  Lovász theta function and its equivalent characterizations and
  variants.
  The main achievement of this paper is the introduction of a new
  family of norms, providing upper bounds for the stability number,
  that are obtained from duality from the norms motivated by Hoffman's
  lower bound for the chromatic number and which achieve the
  (complementary) Lovász theta function at their optimum.
  As a consequence, our norms make it formal that Hoffman's bound for
  the chromatic number and the Delsarte-Hoffman ratio bound for the
  stability number are indeed dual.
  Further, we show that our new bounds strengthen the convex quadratic
  bounds for the stability number studied by Luz and Schrijver, and
  which achieve the Lovász theta function at their optimum.
  One of the key observations regarding weighted versions of these
  bounds is that, for any upper bound for the stability number of a
  graph which is a positive definite monotone gauge function, its
  gauge dual is a lower bound on the fractional chromatic number, and
  conversely.
  Our presentation is elementary and accessible to a wide audience.
\end{abstract}

\maketitle

\section{Introduction}
\label{sec:intro}

Let \(G = (V,E)\) be a graph.
A subset \(S\) of~\(V\) is \emph{stable} if no edge of~\(G\) joins two
vertices of~\(S\).
The \emph{stability number} of~\(G\), denoted by \(\alpha(G)\), is the
maximum size of a stable set in~\(G\).
A \emph{coloring} of~\(G\) is a partition of~\(V\) into stable sets.
The \emph{chromatic number} of~\(G\), denoted by \(\chi(G)\), is the
minimum size of a coloring of~\(G\).
While these classical graph parameters are well known to be NP-hard to
compute, there are several upper bounds for~\(\alpha(G)\) and lower
bounds for~\(\chi(G)\) that work well for important families of
graphs or have other favorable properties.
Many such bounds are spectral, that~is, they arise from the
eigenvalues of matrices associated with the graph~\(G\), such as its
adjacency matrix~\(A_G\).
Recent work on such bounds includes~\cite{Bilu06a, GodsilN08a,
  ElphickW17a}.

Hoffman~\cite{Hoffman70a} proved some of the oldest, most classical
bounds for~\(\alpha\) and~\(\chi\).
One is the so-called \emph{Delsarte-Hoffman ratio bound},
\begin{equation}
  \label{eq:ratio-bound}
  \alpha(G)
  \leq
  \frac{n}{1-{k}/{\tau}},
\end{equation}
which holds for any \(k\)-regular graph \(G\), where \(k \geq 1\) and
\(\tau \coloneqq \lambda_{\min}(A_G)\) is the smallest eigenvalue
of~\(A_G\).
Throughout we write \(n\) for the number of vertices of the (current)
graph.
The other \emph{Hoffman bound} is
\begin{equation}
  \label{eq:Hoffman-bound}
  \chi(G) \geq 1-\frac{\lambda_{\max}(A_G)}{\lambda_{\min}(A_G)},
\end{equation}
which holds for any graph~\(G\) with at least one edge, where
\(\lambda_{\max}\) extracts the largest eigenvalue.
We refer the reader to~\cite[Ch.~3]{Newman04a} for a discussion of
these bounds, including origins and generalizations.

It is interesting to note that many other similar bounds also come
naturally in pairs.
As an example, consider the celebrated graph parameter \(\theta(G)\),
known as the \emph{Lovász theta number} of~\(G\).
This graph parameter was introduced in seminal work of
Lovász~\cite{Lovasz79a} and it can be efficiently computed (to within
any desired precision) by solving a semidefinite program (SDP); we
postpone its exact definition and further references for later.
It~provides both an upper bound for~\(\alpha(G)\) and a lower bound
for~\(\chi(G)\) since
\(\alpha(G) \leq \theta(G) \leq \overline{\chi}(G)\).
Here we are adopting the usual convention of denoting, for every graph
parameter~\(\beta\), the complementary graph
parameter~\(\overline{\beta}\) defined as
\(\overline{\beta}(G) \coloneqq \beta(\overline{G})\), where
\(\overline{G}\) denotes the \emph{complementary graph} of
\(G = (V,E)\), that~is, the graph on~\(V\) whose edges are the
non-edges of~\(G\).
Moreover, \(\theta(G)\overline{\theta}(G) \geq n\), with equality
whenever \(G\) is vertex-transitive.
These are manifestations of the fact that \(\theta\) and
\(\overline{\theta}\) are dual to each~other, in some precise sense.

As another example, consider two variants of the Lovász theta number,
usually denoted by~\(\theta'(G)\) and~\(\theta^+(G)\), introduced
respectively by McEliece, Rodemich, and Rumsey~\cite{McElieceRR78a}
and Schrijver~\cite{Schrijver79a}, and by Szegedy~\cite{Szegedy94a}.
These parameters are obtained from~\(\theta(G)\) by adding/relaxing
constraints from the SDP formulation for~\(\theta(G)\) and they
satisfy
\( \alpha(G) \leq \theta'(G) \leq \theta(G) \leq \theta^+(G) \leq
\overline{\chi}(G) \).
Hence, \(\theta'(G)\) provides an upper bound for~\(\alpha(G)\) and
\(\overline{\theta^+}(G)\) provides a lower bound for~\(\chi(G)\).
Moreover, \(\theta'(G)\overline{\theta^+}(G) \geq n\), and equality
holds if \(G\)~is~vertex-transitive.
As~before, these arise since \(\theta'\) and \(\overline{\theta^+}\)
are dual to each other.

As a final, slightly contrived though crucial example, consider the
trivial upper bound on~\(\alpha(G)\) given by~\(\alpha(G)\) itself,
and the lower
bound \(\chi_f(G)\), known as the \emph{fractional chromatic number},
for~\(\chi(G)\).
It can be defined using a linear program (LP) as follows:
\begin{equation*}
  \chi_f(G)
  \coloneqq
  \min\setst[\bigg]{
    \sum_{S} y_S
  }{
    y \in \Reals_+^{\Scal(G)},\,
    \sum_{S} y_S \incidvector{S} \geq \incidvector{V}
  };
\end{equation*}
both summations range over the set \(\Scal(G)\) of stable sets
of~\(G\), and \(\incidvector{S} \in \set{0,1}^V\) denotes the
incidence vector of \(S \subseteq V\).
We have
\begin{equation}
  \label{eq:sandwich-long}
  \alpha(G)
  \leq
  \theta'(G)
  \leq
  \theta(G)
  \leq
  \theta^+(G)
  \leq
  \overline{\chi_f}(G)
  \leq
  \overline{\chi}(G).
\end{equation}
Once more, \(\alpha(G)\chi_f(G) \geq n\), and equality holds if
\(G\)~is~vertex-transitive.
Again, these are manifestations of \(\alpha\) and~\(\chi_f\) being
dual to each other.
There is a precise, \emph{geometric} notion in which all these pairs
of parameters are dual pairs.

With this context in mind, the ratio bound~\cref{eq:ratio-bound} and
the Hoffman bound~\cref{eq:Hoffman-bound} look suspiciously like a
dual pair.
Note that their product is~\(n\) whenever both bounds apply, which
includes the case where \(G\) is vertex-transitive.
In this paper, we introduce a graph parameter~\(\Hdual\), dual to the
Hoffman bound~\cref{eq:Hoffman-bound}, which:
\begin{enumerate}[(i)]
\item is defined as the optimal value of an SDP;
\item yields the Delsarte-Hoffman ratio bound~\cref{eq:ratio-bound}
  when applied to regular graphs;
\item comes from a family of upper bounds~\(\Hdual_A\)
  for~\(\alpha(G)\) indexed by any generalized adjacency matrix~\(A\)
  of~\(G\), and the best bound in the family coincides
  with~\(\theta(G)\);
\item coincides with a convex quadratic upper bound~\(\Luz(G)\)
  for~\(\alpha(G)\) introduced by Luz~\cite{Luz95a}, and similarly for
  the generalized bounds~\(\Hdual_A\) when the generalized adjacency
  matrix~\(A\) is nonnegative;
\item provides an upper bound on~\(\alpha(G)\), via the dual SDP, that
  depends on the minimum component of the (normalized) Perron
  eigenvector, when \(G\) is connected.
\end{enumerate}
In particular, the new parameter~\(\Hdual\) we introduce, along with
its properties, proves that the
bounds~\cref{eq:ratio-bound,eq:Hoffman-bound} form indeed a dual pair,
according to the precise notion that we shall formalize.

We rely on the remark that any (weighted) upper bound on the stability
number~\(\alpha\) that satisfies some natural properties, which we
call a positive definite monotone gauge, yields via gauge duality a
(weighted) lower bound on the fractional chromatic number~\(\chi_f\),
and vice versa.
These notions come from convex analysis, however our treatment is
self-contained and elementary.

The rest of this paper is organized as follows.
\Cref{sec:gauge-duality} defines positive definite monotone gauges and
lays out the precise notion of duality which links the above pairs.
We introduce our new parameter~\(\Hdual\) in
\cref{sec:dual-Hoffman-bound}, where we prove some of its basic
properties, including that it is dual to the Hoffman
bound~\cref{eq:Hoffman-bound}.
In~\cref{sec:Luz}, we prove that \(\Hdual\) always provides a bound
on~\(\alpha\) at least as good as the bound~\(\Luz\) introduced by
Luz.
We conclude with \cref{sec:theta}, where we prove that the best upper
bound for~\(\alpha(G)\) arising from the family \(\Hdual\) indexed by
generalized adjacency matrices of~\(G\) matches~\(\theta(G)\), as well
as other relationships involving the variants~\(\theta'(G)\)
and~\(\theta^+(G)\).

\section{Duality of Bounds for the Stability and Chromatic Numbers}
\label{sec:gauge-duality}

\subsection{Duality of Norms, Sign-Invariant Norms, and Positive
  Definite Monotone Gauges}

In this section, we present the relevant concepts from the theory of
gauge duality in an accessible form; we~refer the reader to~\cite[§14
and §15]{Rockafellar97a} for a complete treatment.
(Gauge duality has received a lot of attention in the optimization
community recently; see \cite{FriedlanderMP14a,AravkinBDFM18a}.)
We will need to define weighted versions of the stability
number~\(\alpha\), the fractional chromatic number~\(\chi_f\), and
other parameters.
These weighted parameters correspond to linear optimization over
certain convex sets, known as convex corners, which can be thought of
as wedges cut off from unit balls of certain norms.
We will relate convex corners via antiblocking duality, a concept
which in the polyhedral case goes back at least to
Fulkerson~\cite{Fulkerson71a,Fulkerson72a}; see
also~\cite[Sec.~9.3]{Schrijver86a}.
Our development, which grounds the remainder of the text, treads only
on widespread concepts such as norms and their duals, at the cost of
not being the most direct route to the desired results.

Let \(G = (V,E)\) be a graph.
Let \(w \in \Reals_+^V\) be a nonnegative weight function.
Recall that \(\Scal(G)\) denotes the set of stable sets of~\(G\).
The \emph{weighted stability number} of~\(G\) and the \emph{weighted
  fractional chromatic number} of~\(G\) are, respectively,
\begin{gather}
  \alpha(G,w) \coloneqq \max\setst[\big]{
    \iprodt{w}{\incidvector{S}}
  }{
    S \in \Scal(G)
  },
  \notag
  \\
  \label{eq:def-chif}
  \chi_f(G,w) \coloneqq \min\setst[\bigg]{
    \iprodt{\ones}{y}
  }{
    y \in \Reals_+^{\Scal(G)},\,
    \sum_{\mathclap{S \in \Scal(G)}} y_S \incidvector{S} \geq w
  }.
\end{gather}
Here, we denote the vector of all-ones by~\(\ones\).
Combinatorially, that~is, when \(w\) is integer-valued,
\(\alpha(G,w)\) and \(\chi_f(G,w)\) are, respectively, the stability
number and the fractional chromatic number of the graph obtained
from~\(G\) by replacing each vertex~\(i\) by a stable set of
size~\(w_i\).
These parameters correspond to the LPs
\begin{gather}
  \alpha(G,w) = \max\setst[\big]{
    \iprodt{w}{x}
  }{
    x \in \STAB(G)
  },
  \notag
  \\
  \label{eq:chif-over-QSTAB}
  \chi_f(G,w) = \max\setst[\big]{
    \iprodt{w}{x}
  }{
    x \in \QSTAB(\overline{G})
  },
\end{gather}
where
\begin{gather*}
  \STAB(G)
  \coloneqq
  \conv\setst[\big]{
    \incidvector{S} \in \Reals^V
  }{
    S \in \Scal(G)
  },
  \\
  \QSTAB(G)
  \coloneqq
  \setst[\big]{
    x \in \Reals_+^V
  }{
    \iprodt{\incidvector{K}}{x}
    \leq 1
    \,
    \forall K \subseteq V
    \text{ clique in \(G\)}
  };
\end{gather*}
here, \(\conv\) denotes the convex hull.  \Cref{eq:chif-over-QSTAB}
follows from LP Strong Duality.

\emph{Throughout the paper, let \(V\) denote an arbitrary finite set.}

Denote the componentwise absolute value of a vector~\(x \in \Reals^V\)
by \(\abs{x}\).
We will see below that
\begin{equation}
  \label{eq:alpha-chif-yield-norms}
  \text{the functions }
  \norm[\alpha,G]{\cdot} \colon x \in \Reals^V \mapsto \alpha(G,\abs{x})
  \text{ and }
  \norm[\chi_f,G]{\cdot} \colon x \in \Reals^V \mapsto \chi_f(G,\abs{x})
  \text{ are norms on }
  \Reals^V.
\end{equation}
Recall that a norm on~\(\Reals^V\) is a function
\(\norm{\cdot} \colon \Reals^V \to \Reals\) such that
\begin{eqenum}[list:def-norm]
\item\label{item:norm-pd} \(\norm{\cdot}\) is \emph{positive
    definite}, i.e., \(\norm{x} \geq 0\) for every \(x \in \Reals^V\),
  with equality if and only if \(x = 0\);
\item\label{item:norm-homog} \(\norm{\cdot}\) is \emph{absolutely
    homogeneous}, i.e., \(\norm{\lambda x} = \abs{\lambda}\norm{x}\)
  for every scalar \(\lambda \in \Reals\) and every
  \(x \in \Reals^V\);
\item\label{item:norm-triangle} \(\norm{\cdot}\) satisfies the
  \emph{triangle inequality}, i.e.,
  \(\norm{x+y} \leq \norm{x} + \norm{y}\) for every
  \(x,y \in \Reals^V\).
\end{eqenum}
Let \(\Ball \coloneqq \setst{x \in \Reals^V}{\norm{x} \leq 1}\) be the
\emph{unit ball} of the norm~\(\norm{\cdot}\).
The \emph{dual norm} of~\(\norm{\cdot}\) is the function
\(\norm{\cdot}^* \colon \Reals^V \to \Reals\) defined as
\begin{equation}
  \label{eq:def-dual-norm}
  \norm{y}^*
  \coloneqq
  \max\setst{
    \iprodt{x}{y}
  }{
    \norm{x} \leq 1
  }
  =
  \max_{x \in \Ball} \iprodt{x}{y},
  \qquad
  \forall y \in \Reals^V.
\end{equation}
Recall that the \emph{polar} of a set \(\Xcal \subseteq \Reals^V\) is
\begin{equation*}
  \Xcal^{\polar}
  \coloneqq
  \setst{
    y \in \Reals^V
  }{
    \iprodt{x}{y} \leq 1,\forall x \in \Xcal
  }.
\end{equation*}
The \emph{Minkowski functional} of~\(\Xcal \subseteq \Reals^V\) is
\begin{equation}
  \label{eq:def-Minkowski-functional}
  \gamma_{\Xcal}(x)
  \coloneqq
  \inf\setst{
    \mu \in \Reals_+
  }{
    x \in \mu\Xcal
  },
  \qquad
  \forall x \in \Reals^V.
\end{equation}
The next result isolates which properties of~\(\Ball\) make it the
unit ball of some norm, and which ensure that \(\norm{\cdot}^*\) is a
norm:
\begin{proposition}[Construction of Norms]
  \label{prop:norm-construction}
  Let \(\Ball \subseteq \Reals^V\) be a compact convex set having
  \(0\) in its interior and such that \(\Ball = -\Ball\).  Then:
  \begin{eqenum}[list:norm-construction]
  \item\label{item:norm-construction-primal} the function
    \(\norm{\cdot} \coloneqq \gamma_{\Ball}(\cdot)\) is a norm
    on~\(\Reals^V\) with unit ball~\(\Ball\);
  \item\label{item:norm-construction-dual} the function
    \(\norm{\cdot}^* \colon y \in \Reals^V \mapsto \max_{x \in \Ball}
    \iprodt{x}{y}\) is a norm on~\(\Reals^V\) with unit
    ball~\(\Ball^{\polar}\).
  \end{eqenum}
\end{proposition}
\begin{proof}
  \Cref{item:norm-construction-primal}:
  Clearly \(\norm{x} \geq 0\) for every \(x \in \Reals^V\) and
  \(\norm{0} = 0\).
  If \(x \in \Reals^V\) is nonzero and \(\mu \in \Reals_+\), then
  \begin{equation}
    \label{eq:gauge-rescaling}
    x \in \mu\Ball
    \iff
    \mu > 0
    \text{ and }
    \tfrac{x}{\mu} \in \Ball.
  \end{equation}
  Since \(\Ball\) is bounded, there is \(\eps > 0\) such that
  \(\gamma_{\Ball}(x) = \inf\setst{\mu \geq \eps}{\tfrac{x}{\mu} \in
    \Ball}\).
  In~particular, the `\(\inf\)' in \cref{eq:def-Minkowski-functional}
  is attained by compactness of~\(\Ball\), and \(\norm{\cdot}\) is
  positive definite.
  Absolute homogeneity of~\(\norm{\cdot}\) follows from
  \(\Ball=-\Ball\).
  For the triangle inequality, let \(x,y \in \Reals^V\) be nonzero,
  and set \(\mu \coloneqq \norm{x} > 0\) and
  \(\eta \coloneqq \norm{y} > 0\).
  Since the `\(\inf\)' is attained, we have from
  \cref{eq:gauge-rescaling} that
  \(\tfrac{x}{\mu},\tfrac{y}{\eta} \in \Ball\).
  Finally, since \(\Ball\) is convex,
  \(\tfrac{x+y}{\mu+\eta} = \tfrac{\mu}{\mu+\eta} \tfrac{x}{\mu} +
  \tfrac{\eta}{\mu+\eta} \tfrac{y}{\eta} \in \Ball\), whence
  \(\norm{x+y} = \gamma_{\Ball}(x+y) \leq \mu+\eta =
  \norm{x}+\norm{y}\).
  Hence, \(\norm{\cdot}\) is a norm, and its unit ball is~\(\Ball\)
  since the `\(\inf\)' is always attained.

  \Cref{item:norm-construction-dual}:
  Attainment in the~`\(\max\)' follows from compactness of~\(\Ball\).
  Positive definiteness is a consequence of \(0\) being in the
  interior of~\(\Ball\).
  Absolute homogeneity follows from \(\Ball = -\Ball\), and the
  triangle inequality follows from linearity and basic properties of
  `\(\max\)'.
  Hence, \(\norm{\cdot}^*\) is a norm, and its unit ball is
  \(\setst{y \in \Reals^V}{\norm{y}^* \leq 1} = \setst{y \in
    \Reals^V}{\iprodt{x}{y} \leq 1,\forall x \in \Ball} =
  \Ball^{\polar}\).
\end{proof}

We can now state the duality properties of norms:
\begin{theorem}[Norm Duality]
  \label{thm:norm-duality}
  Let \(\norm{\cdot}\) be a norm on~\(\Reals^V\) with unit
  ball~\(\Ball\).
  Then:
  \begin{eqenum}[list:norm-duality]
  \item\label{item:norm-unit-ball} \(\Ball\) is a compact convex set
    having \(0\) in its interior and \(\Ball = -\Ball\);
  \item\label{item:dual-norm-closed} \(\norm{\cdot}^*\) is a norm with
    unit ball~\(\Ball^{\polar}\);
  \item\label{item:norm-duality}
    \(\norm{\cdot}^{**} = \norm{\cdot}\) and, equivalently,
    \(\Ball^{\polar\polar} = \Ball\);
  \item\label{item:norm-cauchy}
    \(\iprodt{x}{y} \leq \norm{x}\norm{y}^*\) for every
    \(x,y \in \Reals^V\).
  \end{eqenum}
\end{theorem}
\begin{proof}
  \Cref{item:norm-unit-ball}:
  Absolute homogeneity shows that \(\Ball = -\Ball\).
  Absolute homogeneity and the triangle inequality show that
  \(\norm{\cdot}\) is a convex function, and thus continuous.
  Hence, \(0\) is in the interior of~\(\Ball\).
  Since \(\Ball = \setst{x \in \Reals^V}{\norm{x}\leq 1}\) is a
  sub-level set of~\(\norm{\cdot}\), it is both convex and closed.
  Set
  \(\mu \coloneqq \min\setst{\norm{x}}{x \in \Reals^V,\, \norm[2]{x} =
    1} > 0\), where \(\norm[2]{\cdot}\) is the usual euclidean norm.
  Then \(\mu\norm[2]{x} \leq \norm{x}\) for every \(x \in \Reals^V\).
  Hence, \(\Ball\) is bounded (with respect to the euclidean norm).

  \Cref{item:dual-norm-closed}:
  Immediate from \cref{item:norm-unit-ball,item:norm-construction-dual}.

  \Cref{item:norm-duality}:
  It is a well known consequence of the Hahn-Banach Theorem that the
  dual of the dual norm~\(\norm{\cdot}^*\) is the original
  norm~\(\norm{\cdot}\), i.e., \(\norm{\cdot}^{**} = \norm{\cdot}\);
  see, e.g., \cite[Ch.~IV,Prop.~1.3]{Lang93a}.
  Hence, its unit ball~\(\Ball\) equals \(\Ball^{\polar\polar}\)
  by~\cref{item:dual-norm-closed}.

  \Cref{item:norm-cauchy}:
  If \(x = 0\), there is nothing to prove.
  If \(x \neq 0\) then for
  \(\lambda \coloneqq \norm{x} > 0\) we have
  \(\norm{\lambda^{-1}x}=1\), whence \(\iprodt{(\lambda^{-1}x)}{y}
  \leq \norm{y}^*\).
\end{proof}

By \cref{item:norm-duality,item:dual-norm-closed}, if \(\norm{\cdot}\)
is a norm with unit ball~\(\Ball\), then
\(\norm{x} = \norm{x}^{**} = \max_{y \in \Ball^{\polar}}
\iprodt{x}{y}\).
Hence, \cref{item:norm-unit-ball,item:dual-norm-closed} show that
\emph{every} norm arises as in the construction in
\cref{item:norm-construction-dual}.
Moreover, whenever \(\Ball \subseteq \Reals^V\) satisfies the
hypotheses of \cref{prop:norm-construction}, \(\Ball\) is the unit
ball of some norm by \cref{item:norm-construction-primal}, whence
\(\Ball^{\polar\polar} = \Ball\) by \cref{item:norm-duality}.

Many norms are \emph{sign-invariant}, that~is,
\(\norm{x} = \norm{\abs{x}}\) for every \(x \in \Reals^V\).
Note that each \(p\)-norm
\(\norm[p]{x} \coloneqq \paren{\sum_{i \in V}\abs{x_i}^p}^{1/p}\),
with real \(p \geq 1\), is sign-invariant, and so is the
\(\infty\)-norm
\(\norm[\infty]{x} \coloneqq \max_{i \in V}\abs{x_i}\).
The norms \(\norm[\alpha,G]{\cdot}\) and \(\norm[\chi_f,G]{\cdot}\)
from~\cref{eq:alpha-chif-yield-norms} are also sign-invariant by
definition.
Let us call \(\Xcal \subseteq \Reals^V\) \emph{sign-symmetric} if, for
every \(x \in \Reals^V\), we have \(x \in \Xcal\) if and only if
\(\abs{x} \in \Xcal\).
Then the unit ball~\(\Ball\) of a sign-invariant norm \(\norm{\cdot}\)
is sign-symmetric, whence all the information encoded in~\(\Ball\) is
contained in the wedge
\begin{equation*}
  \Ccal \coloneqq \Ball \cap \Reals_+^V
\end{equation*}
of~\(\Ball\) that lies in the nonnegative orthant; that~is, we can
recover~\(\Ball\) from~\(\Ccal\), since
\begin{equation*}
  \Ball = \setst{x \in \Reals^V}{\abs{x} \in \Ccal}.
\end{equation*}
We will prove below (see~\cref{item:gauge-corner}) that \(\Ccal\) is a
\emph{convex corner}, i.e., \(\Ccal\) is a lower-comprehensive compact
convex set with nonempty interior that lies in the nonnegative orthant
\(\Reals_+^V\).
We say that \(\Xcal \subseteq \Reals_+^V\) is
\emph{lower-comprehensive} if, whenever \(0 \leq x \leq y \in \Xcal\),
we have \(x \in \Xcal\).
We call \(\Ccal = \setst{x \in \Reals_+^V}{\norm{x} \leq 1}\) the
\emph{unit convex corner} of the sign-invariant norm~\(\norm{\cdot}\).
The \emph{antiblocker} of \(\Xcal \subseteq \Reals_+^V\) is
\begin{equation*}
  \abl(\Xcal) \coloneqq \Xcal^{\polar} \cap \Reals_+^V.
\end{equation*}

The unit convex corner of the sign-invariant norm
\(\norm[\alpha,G]{\cdot}\) from~\cref{eq:alpha-chif-yield-norms} is
\begin{equation*}
  \setst{
    x \in \Reals_+^V
  }{
    \alpha(G,x) \leq 1
  }
  =
  \setst{
    x \in \Reals_+^V
  }{
    \iprodt{\incidvector{S}}{x} \leq 1,
    \forall S \in \Scal(G)
  }
  =
  \abl(\STAB(G)),
\end{equation*}
and the unit convex corner of the sign-invariant norm
\(\norm[\chi_f,G]{\cdot}\) is
\begin{equation*}
  \setst{
    x \in \Reals_+^V
  }{
    \chi_f(G,x) \leq 1
  }
  =
  \setst{
    x \in \Reals_+^V
  }{
    \iprodt{x}{y} \leq 1,
    \forall y \in \QSTAB(\overline{G})
  }
  =
  \abl(\QSTAB(\overline{G})).
\end{equation*}

The next result shows how to construct sign-invariant norms from
sign-symmetric sets; it is the sign-invariant counterpart to
\cref{prop:norm-construction}.
\begin{proposition}[Construction of Sign-Invariant Norms]
  \label{prop:snorm-construction}
  Let \(\Ball \subseteq \Reals^V\) be a sign-symmetric compact convex
  set having \(0\) in its interior.  Then:
  \begin{eqenum}[list:snorm-construction]
  \item\label{item:snorm-construction-primal} the function
    \(\norm{\cdot} \coloneqq \gamma_{\Ball}(\cdot)\) is a
    sign-invariant norm on~\(\Reals^V\) with unit ball~\(\Ball\);
  \item\label{item:snorm-construction-dual} the function
    \(\norm{\cdot}^* \colon y \in \Reals^V \mapsto \max_{x \in \Ball}
    \iprodt{x}{y}\) is a sign-invariant norm on~\(\Reals^V\) with unit
    ball~\(\Ball^{\polar}\).
  \end{eqenum}
\end{proposition}
\begin{proof}
  Since \(\Ball\) is sign-symmetric, we have \(\Ball = -\Ball\).
  \Cref{item:snorm-construction-primal}:
  By \cref{item:norm-construction-primal}, \(\norm{\cdot}\) is a norm
  on~\(\Reals^V\) with unit ball~\(\Ball\).
  If \(x \in \Reals^V\), then
  \(\norm{x} = \inf\setst{\mu \in \Reals_+}{x \in \mu\Ball} =
  \inf\setst{\mu \in \Reals_+^V}{\abs{x} \in \mu\Ball} =
  \norm{\abs{x}}\), where the middle equation follows from the
  sign-symmetry of~\(\Ball\).
  Hence, \(\norm{\cdot}\) is sign-invariant.

  \Cref{item:snorm-construction-dual}:
  By \cref{item:norm-construction-dual},
  \(\norm{\cdot}^* \colon y \in \Reals^V \mapsto \max_{x \in \Ball}
  \iprodt{x}{y}\) is a norm with unit ball~\(\Ball^{\polar}\).
  It remains to show that \(\norm{\cdot}^*\) is sign-invariant.
  Let \(y \in \Reals^V\).
  Then
  \(\norm{y}^* = \max_{x \in \Ball} \iprodt{x}{y} = \max_{x \in \Ball}
  \iprodt{x}{\abs{y}} = \norm{\abs{y}}^*\) since, by the sign-symmetry
  of~\(\Ball\), the leftmost \(\max\) is attained by some~\(x\) whose
  components have signs matching those of~\(y\).
\end{proof}

Now we can state the duality results for sign-invariant norms, the
sign-invariant counterpart to \cref{thm:norm-duality}.
\begin{theorem}[Duality of Sign-Invariant Norms]
  \label{thm:sign-norm-duality}
  Let \(\norm{\cdot}\) be a sign-invariant norm on~\(\Reals^V\) with
  unit ball~\(\Ball\).
  Then:
  \begin{eqenum}[list:sign-norm-duality]
  \item\label{item:snorm-unit-sball} \(\Ball\) is a sign-symmetric compact
    convex set having \(0\) in its interior;
  \item\label{item:dual-snorm-closed} \(\norm{\cdot}^*\) is a
    sign-invariant norm with unit ball~\(\Ball^{\polar}\);
  \item\label{item:snorm-duality} \(\norm{\cdot}^{**} = \norm{\cdot}\)
    and, equivalently, \(\Ball^{\polar\polar} = \Ball\);
  \item\label{item:snorm-cauchy}
    \(\iprodt{x}{y} \leq \norm{x}\norm{y}^*\) for every
    \(x,y \in \Reals^V\).
  \end{eqenum}
\end{theorem}
\begin{proof}
  \Cref{item:snorm-unit-sball}:
  By \cref{item:norm-unit-ball}, it suffices to prove that~\(\Ball\)
  is sign-symmetric.
  However, this is immediate since \(\norm{\cdot}\) is sign-invariant.

  \Cref{item:dual-snorm-closed}:
  Immediate from \cref{item:snorm-unit-sball,item:snorm-construction-dual}.

  \Cref{item:snorm-duality,item:snorm-cauchy}:
  Immediate from \cref{item:norm-duality,item:norm-cauchy}, respectively.
\end{proof}

As before, note that
\cref{item:snorm-unit-sball,item:dual-snorm-closed,item:snorm-duality}
imply that \emph{every} sign-invariant norm arises as in the
construction in \cref{item:snorm-construction-dual}.

The next few definitions capture the relevant properties of the
restriction of (sign-invariant) norms to the nonnegative orthant.
A function \(\kappa \colon \Reals_+^V \to \Reals\) is a \emph{gauge}
if
\begin{eqenum}[list:def-gauge]
\item \(\kappa\) is \emph{positive semidefinite}, i.e.,
  \(\kappa(w) \geq 0\) for every \(w \in \Reals_+^V\) and
  \(\kappa(0) = 0\);
\item \(\kappa\) is \emph{positively homogeneous}, i.e.,
  \(\kappa(\lambda w) = \lambda \kappa(w)\) for every scalar
  \(\lambda > 0\) and \(w \in \Reals_+^V\);
\item \(\kappa\) is \emph{sublinear}, i.e.,
  \(\kappa(w + z) \leq \kappa(w) + \kappa(z)\) for every
  \(w,z \in \Reals_+^V\).
\end{eqenum}
A gauge \(\kappa\) is \emph{positive definite} if \(\kappa(w) > 0\)
whenever \(w \in \Reals_+^V\) is nonzero, and \(\kappa\) is
\emph{monotone} if \(\kappa(w) \leq \kappa(z)\) whenever
\(w,z \in \Reals_+^V\) satisfy \(w \leq z\).
Our main interest in positive definite monotone gauges arises from the
easily verified fact that
\begin{equation}
  \label{eq:alpha-chif-pdmgs}
  \text{%
    the functions \(\alpha(G,\cdot)\) and \(\chi_f(G,\cdot)\)
    on~\(\Reals_+^V\) are positive definite monotone gauges.
  }
\end{equation}

The exact connection between sign-invariant norms and positive
definite monotone gauges, which we will use to translate their duality
theories, is subsumed by the following constructions:
\begin{eqenum}[list:snorm-gauge]
\item\label{item:gauge-from-snorm} the restriction to~\(\Reals_+^V\)
  of a sign-invariant norm \(\norm{\cdot}\) on~\(\Reals^V\) is a
  positive definite monotone gauge;
\item\label{item:snorm-from-gauge} if
  \(\kappa \colon \Reals_+^V \to \Reals\) is a positive definite
  monotone gauge, then
  \(\norm[\kappa]{\cdot} \colon x \in \Reals^V \mapsto
  \kappa(\abs{x})\) is a sign-invariant norm.
\end{eqenum}

We shall rely on the following key property: if
\(\Xcal \subseteq \Reals^V\) is a convex and sign-symmetric set, then
\begin{equation}
  \label{eq:sign-symmetric-closure}
  \Xcal \supseteq
  \conv\setst{x^s}{s \in \set{\pm1}^V},
  \qquad
  \forall x \in \Xcal,
\end{equation}
where \(x^s \in \Reals^V\) is defined as \(x^s(i) \coloneqq s(i)x(i)\)
for every \(i \in V\) and \(s \in \set{\pm1}^V\).

To prove \cref{item:gauge-from-snorm}, start by noting that the
restriction to~\(\Reals_+^V\) of any norm is a positive definite
gauge.
Next we prove monotonicity.
Suppose that \(\norm{\cdot}\) is sign-invariant with unit
ball~\(\Ball\).
By \cref{item:snorm-unit-sball,item:snorm-construction-primal}, we
have \(\norm{\cdot} = \gamma_{\Ball}(\cdot)\).
Let \(x,y \in \Reals_+^V\) with \(x \leq y\).
If \(\mu \in \Reals_+\) satisfies \(y \in \mu\Ball\), then by
\cref{eq:sign-symmetric-closure} we get
\(\mu\Ball \supseteq \conv\setst{y^{s}}{s \in \set{\pm1}^V} \ni x\).
Hence,
\(\norm{x} = \gamma_{\Ball}(x) \leq \gamma_{\Ball}(y) = \norm{y}\).
This completes the proof of \cref{item:gauge-from-snorm}.

It remains to prove \cref{item:snorm-from-gauge}.
\Cref{item:norm-pd} follows from positive definiteness of~\(\kappa\)
and \cref{item:norm-homog} follows from positive homogeneity
of~\(\kappa\).
For \cref{item:norm-triangle}, if \(x,y \in \Reals^V\), then
\(\norm[\kappa]{x+y} = \kappa(\abs{x+y}) \leq \kappa(\abs{x}+\abs{y})
\leq \kappa(\abs{x}) + \kappa(\abs{y}) = \norm[\kappa]{x} +
\norm[\kappa]{y}\) by monotonicity with the triangle inequality
\(\abs{x+y} \leq \abs{x}+\abs{y}\), and by sublinearity of~\(\kappa\).
It is obvious that \(\norm[\kappa]{\cdot}\) is sign-invariant.
This completes the proof of~\cref{item:snorm-from-gauge}.

Note that \cref{eq:alpha-chif-pdmgs,item:snorm-from-gauge} finally
prove \cref{eq:alpha-chif-yield-norms}.
The \emph{unit convex corner} of a positive definite monotone
gauge~\(\kappa\) is the unit convex corner of the norm
\(\norm[\kappa]{\cdot}\) defined in~\cref{item:snorm-from-gauge}.
We can now state the gauge counterpart to
\cref{prop:norm-construction,prop:snorm-construction}.
\begin{proposition}[Construction of Gauges]
  \label{prop:gauge-construction}
  Let \(\Ccal \subseteq \Reals_+^V\) be a convex corner.  Then:
  \begin{eqenum}[list:gauge-construction]
  \item\label{item:gauge-construction-primal} the function
    \(x \in \Reals_+^V \mapsto \gamma_{\Ccal}(x)\) is a positive
    definite monotone gauge with unit convex corner~\(\Ccal\);
  \item\label{item:gauge-construction-dual} the function
    \(y \in \Reals_+^V \mapsto \max_{x \in \Ccal} \iprodt{x}{y}\) is a
    positive definite monotone gauge with unit convex
    corner~\(\abl(\Ccal)\).
  \end{eqenum}
\end{proposition}
\begin{proof}
  \Cref{item:gauge-construction-primal}:
  Define
  \(\Ball \coloneqq \setst{x \in \Reals^V}{\abs{x} \in \Ccal}\).
  We claim that \(\Ball\) is a sign-symmetric compact convex set
  having \(0\) in its interior.
  It is clear that \(\Ball\) is sign-symmetric and compact.
  Since \(\Ccal\) is lower-comprehensive and has nonempty interior,
  there is \(\eps > 0\) such that \(\eps\ones \in \Ccal\).
  Together with sign-symmetry of~\(\Ball\), this shows that \(0\) lies
  in the interior of~\(\Ball\).
  It remains to prove convexity of~\(\Ball\).
  Let \(x,y \in \Ball\), and let \(\lambda \in [0,1]\).
  Then
  \(\abs{\lambda x + (1-\lambda)y} \leq \lambda\abs{x} +
  (1-\lambda)\abs{y}\) by the triangle inequality, and the RHS lies
  in~\(\Ccal\) by convexity.
  Since \(\Ccal\) is lower-comprehensive, we get
  \(\abs{\lambda x + (1-\lambda)y} \in \Ccal\), whence
  \(\lambda x + (1-\lambda)y \in \Ball\).
  By \cref{item:snorm-construction-primal},
  \(\norm{\cdot} \coloneqq \gamma_{\Ball}(\cdot)\) is a sign-invariant
  norm with unit ball~\(\Ball\), whence its restriction
  to~\(\Reals_+^V\) is a positive definite monotone gauge by
  \cref{item:gauge-from-snorm}.

  \Cref{item:gauge-construction-dual}:
  Define
  \(\eta \colon y \in \Reals_+^V \mapsto \max_{x \in \Ccal}
  \iprodt{x}{y}\); the `\(\max\)' is attained (and thus real-valued)
  by compactness of~\(\Ccal\).
  Since \(\Ccal \subseteq \Reals_+^V\) is lower comprehensive and has
  nonempty interior, for every \(y \in \Reals_+^V\), there exists
  \(\eps > 0\) such that \(\eps y \in \Ccal\).
  Hence, \(\eta\) is positive definite.
  Clearly, \(\eta\) is positively homogeneous and sublinear.
  Thus, \(\eta\) is a positive definite gauge.
  If \(y,z \in \Reals_+^V\) satisfy \(y \leq z\) and \(x \in \Ccal\)
  attains the `\(\max\)' in \(\eta(y)\), then
  \(\eta(y) = \iprodt{x}{y} \le \iprodt{x}{z} \le \eta(z)\) since
  \(x \geq 0\).
  In other words, \(\eta\) is monotone.
  The unit convex corner of~\(\eta\) is
  \(\setst{y \in \Reals_+^V}{\eta(y) \leq 1} = \setst{y \in
    \Reals_+^V}{\iprodt{x}{y} \leq 1,\forall x \in \Ccal} =
  \abl(\Ccal)\).
\end{proof}

Let \(\kappa \colon \Reals_+^V \to \Reals\) be a positive definite
monotone gauge.
The \emph{dual (gauge)} of \(\kappa\) is the function
\(\kappa^{\polar} \colon \Reals_+^V \to \Reals\) defined by
\begin{equation}
  \label{eq:def-polar-gauge}
  \kappa^{\polar}(z)
  \coloneqq
  \max\setst{
    \iprodt{w}{z}
  }{
    w \in \Reals_+^V,\,
    \kappa(w) \leq 1
  },
  \qquad
  \forall z \in \Reals_+^V.
\end{equation}
(We do not adopt the more parallel notation \(\kappa^*\) because, in
the convex analysis literature, \(\kappa^*\) stands for the
``conjugate'' of the function~\(\kappa\), a related though different
notion of dual object.)
The next result shows how the duality properties from sign-invariant
norms in \cref{thm:sign-norm-duality} translate to dual gauges.
\begin{theorem}[Gauge Duality]
  \label{thm:gauge-duality}
  Let \(\kappa \colon \Reals_+^V \to \Reals\) be a positive definite
  monotone gauge.
  Let \(\Ball\) be the unit ball of the sign-invariant
  norm~\(\norm[\kappa]{\cdot}\), and set
  \(\Ccal \coloneqq \Ball \cap \Reals_+^V\).
  Then:
  \begin{eqenum}[list:gauge-duality]
  \item\label{item:gauge-corner} \(\Ccal\) is a convex corner.
  \item\label{item:polar-gauge-closed} \(\kappa^{\polar}\) is a
    positive definite monotone gauge with unit convex corner
    \(\abl(\Ccal)\);
  \item\label{item:gauge-duality} \(\kappa^{\polar\polar} = \kappa\)
    and, equivalently, \(\abl(\abl(\Ccal)) = \Ccal\);
  \item\label{item:gauge-cauchy}
    \(\iprodt{w}{z} \leq \kappa(w) \kappa^{\polar}(z)\) for every
    \(w,z \in \Reals_+^V\).
  \end{eqenum}
\end{theorem}
\begin{proof}
  \Cref{item:gauge-corner}:
  Clearly \(\Ccal \subseteq \Reals_+^V\).
  Since \(\Ball\) is compact and
  convex by \cref{item:snorm-unit-sball}, so is~\(\Ccal\).
  Also, \(\Ccal\) has nonempty interior since \(0\) lies in the
  interior of~\(\Ball\) by \cref{item:snorm-unit-sball}.
  It remains to prove that \(\Ccal\) is lower-comprehensive.
  Let \(y \in \Ccal\).
  By \cref{item:snorm-unit-sball,eq:sign-symmetric-closure}, we have
  \(\Ball \supseteq \conv\setst{y^s}{s \in \set{\pm1}^V} \supseteq
  \setst{x \in \Reals_+^V}{x \leq y}\).
  Hence, every \(x \in \Reals_+^V\) such that \(x \leq y\) also lies
  in~\(\Ccal\).

  \Cref{item:polar-gauge-closed}:
  Immediate from \cref{item:gauge-corner,item:gauge-construction-dual}.

  \Cref{item:gauge-duality}:
  By \cref{item:polar-gauge-closed,item:snorm-from-gauge},
  \(\norm[\kappa^{\polar}]{\cdot} \coloneqq
  \kappa^{\polar}(\abs{\cdot})\) is a sign-invariant norm.
  We claim that
  \begin{equation}
    \label{eq:dual-norm-and-gauge-commute}
    \norm[\kappa^{\polar}]{\cdot} = \norm[\kappa]{\cdot}^*.
  \end{equation}
  For every \(y \in \Reals^V\), we have
  \(
  \norm[\kappa^{\polar}]{y}
  =
  \kappa^{\polar}(\abs{y})
  =
  \max_{x \in \Ccal} \iprodt{x}{\abs{y}}
  =
  \max_{x \in \Ball} \iprodt{x}{\abs{y}}
  =
  \norm[\kappa]{\abs{y}}^*
  =
  \norm[\kappa]{y}^*
  \).
  This also shows that
  \(\norm[\kappa^{\polar\polar}]{\cdot} = \norm[\kappa]{\cdot}^{**}\),
  which equals \(\norm[\kappa]{\cdot}\) by \cref{item:snorm-duality}.
  Hence, for every \(w \in \Reals_+^V\), we have
  \(\kappa^{\polar\polar}(w) = \norm[\kappa^{\polar\polar}]{w} =
  \norm[\kappa]{w}^{**} = \norm[\kappa]{w} = \kappa(w)\).
  The unit convex corners of~\(\kappa\) and \(\kappa^{\polar\polar}\)
  are~\(\Ccal\) and \(\abl(\abl(\Ccal))\), respectively; the latter
  follows from \cref{item:polar-gauge-closed}.
  Since \(\kappa = \kappa^{\polar\polar}\), these unit convex corners
  are the same.

  \Cref{item:gauge-cauchy}:
  Let \(w, z \in \Reals_+^V\).
  By \cref{item:snorm-cauchy,eq:dual-norm-and-gauge-commute},
  \(\iprodt{w}{z} \leq \norm[\kappa]{w}\norm[\kappa]{z}^*
  = \norm[\kappa]{w}\norm[\kappa^{\polar}]{z}
  = \kappa(w)\kappa^{\polar}(z)\).
\end{proof}

Once again,
\cref{item:gauge-corner,item:polar-gauge-closed,item:gauge-duality}
show that \emph{every} positive definite monotone gauge arises as in
the construction of \cref{item:gauge-construction-dual}.  Moreover,
every convex corner is the unit convex corner of some positive
definite monotone gauge by \cref{item:gauge-construction-primal} and
so \cref{item:gauge-duality} yields that
\begin{equation}
  \label{eq:abl-abl}
  \abl\paren{
    \abl\paren{
      \Ccal
    }
  }
  =
  \Ccal
  \qquad
  \text{for every convex corner \(\Ccal \subseteq \Reals_+^V\)}.
\end{equation}

\subsection{Duality of Bounds and Graph Parameters}

Let \(G = (V,E)\) be a graph, and let
\(\beta_G \colon \Reals_+^V \to \Reals\) be an upper bound on the
weighted stability number, i.e., \(\alpha(G,w) \leq \beta_G(w)\) for
every \(w \colon V \to \Reals_+\).
Note that, if \(\beta_G\) is a monotone gauge, then it is positive
definite, since it is lower bounded by the positive definite monotone
gauge \(\alpha(G,\cdot)\); see~\cref{eq:alpha-chif-pdmgs}.
We proceed to prove that the dual of~\(\beta_G\) yields a lower bound
for \(\chi_f(G,\cdot)\), i.e.,
\(\beta_G^{\polar}(w) \leq \chi_f(G,w)\) for every
\(w \in \Reals_+^V\).
See~\cite{GvozdenovicL08a} for related work.

As a first step, notice that duality reverses inclusions and
inequalities, as usual:
\begin{lemma}
  \label{le:polar-gauge-reversal}
  Let \(\kappa\) and~\(\eta\) be positive definite monotone gauges
  on~\(\Reals_+^V\).
  Then \(\kappa(w) \leq \eta(w)\) for every \(w \in \Reals_+^V\) if
  and only if \(\eta^{\polar}(w) \leq \kappa^{\polar}(w)\) for every
  \(w \in \Reals_+^V\).
\end{lemma}
\begin{proof}
  By \cref{item:gauge-duality}, it suffices to prove `only~if'.
  Suppose that \(\kappa(w) \leq \eta(w)\) for every
  \(w \in \Reals_+^V\).
  If \(w \in \Reals_+^V\) satisfies \(\eta(w) \leq 1\), then
  \(\kappa(w) \leq \eta(w) \leq 1\), so the feasible region of the
  optimization problem~\cref{eq:def-polar-gauge}
  defining~\(\kappa^{\polar}\) contains the feasible region
  defining~\(\eta^{\polar}\).
\end{proof}

Next we show that \(\alpha\) and \(\chi_f\) are duals, and similarly
that \(\STAB(G)\) and \(\QSTAB(\overline{G})\) are antiblockers of
each other.
This goes back to the work of Fulkerson; see
\cite[Sec.~9.3]{Schrijver86a}.
We begin with a simple observation:
\begin{equation}
  \label{eq:unit-convex-corner-alpha}
  \abl(\STAB(G))
  = \setst{
    x \in \Reals_+^V
  }{
    \iprodt{\incidvector{S}}{x} \leq 1,
    \forall S \in \Scal(G)
  }
  = \QSTAB(\overline G).
\end{equation}
\begin{theorem}
  \label{thm:alpha-chif-duality}
  Let \(G = (V,E)\) be a graph.
  Then the dual of the positive definite monotone gauge
  \(\alpha(G,\cdot)\) is \(\chi_f(G,\cdot)\), and the dual of the
  positive definite monotone gauge \(\chi_f(G,\cdot)\) is
  \(\alpha(G,\cdot)\).
  Moreover, the antiblocker of~\(\STAB(G)\)
  is~\(\QSTAB(\overline{G})\), and the antiblocker of
  \(\QSTAB(\overline{G})\) is \(\STAB(G)\).
\end{theorem}
\begin{proof}
  Denote \(\alpha_G \colon w \in \Reals_+^V \mapsto \alpha(G,w)\).
  Then, by~\cref{eq:unit-convex-corner-alpha}, for every
  \(z \in \Reals_+^V\),
  \begin{equation*}
    \begin{split}
      \alpha_G^{\polar}(z)
      & =
      \max\setst{
        \iprodt{w}{z}
      }{
        w \in \Reals_+^V,\,
        \alpha_G(w) \leq 1
      }
      \\
      & =
      \max\setst{
        \iprodt{w}{z}
      }{
        w \in \Reals_+^V,\,\iprodt{\incidvector{S}}{w} \le 1,
        \forall S \in \Scal(G)
      }
      \\
      & =
      \max\setst{
        \iprodt{w}{z}
      }{
        w \in \QSTAB(\overline{G})
      }
      =
      \chi_f(G,z).
    \end{split}
  \end{equation*}
  By~\cref{eq:alpha-chif-pdmgs,item:gauge-duality}, we get
  \(\chi_f(G,\cdot)^{\polar} = \alpha_G^{\polar\polar} = \alpha_G\).

  By \cref{eq:unit-convex-corner-alpha},
  \(\abl(\STAB(G)) = \QSTAB(\overline{G})\).
  Hence,
  \(\abl(\QSTAB(\overline{G})) = \abl(\abl(\STAB(G))) = \STAB(G)\)
  by~\cref{eq:abl-abl}.
\end{proof}

Since \(\norm[1]{\cdot}\) and \(\norm[\infty]{\cdot}\) are dual
sign-invariant norms, their restrictions to the nonnegative
orthant~\(\Reals_+^V\) are positive definite monotone gauges which are
dual to each other, by \cref{item:gauge-from-snorm}.
Trivially, for every graph \(G = (V,E)\) we have both that
\(\alpha(G,w) \leq \norm[1]{w}\) and
\(\norm[\infty]{z} \leq \chi_f(G,z)\) for every
\(w,z \in \Reals_+^V\).
\Cref{le:polar-gauge-reversal,thm:alpha-chif-duality} imply that both
inequalities are equivalent by duality.
Furthermore, as the next theorem alludes to, our work will focus only
on bounds which are at least as tight as the ones just mentioned.

\begin{theorem}
  \label{thm:bound-conversion}
  Let \(G = (V,E)\) be a graph.
  Let \(\beta_G \colon \Reals_+^V \to \Reals\) be a positive definite
  monotone gauge.
  Then:
  \begin{eqenum}[list:bound-duality]
  \item\label{item:polar-bound-closed} \(\beta_G^{\polar}\) is a
    positive definite monotone gauge;
  \item\label{item:bound-duality} \(\beta_G^{\polar\polar} = \beta_G\);
  \item\label{item:alpha-to-chi} if
    \(\alpha(G,w) \leq \beta_G(w) \leq \norm[1]{w}\) for every
    \(w \in \Reals_+^V\), then
    \(\norm[\infty]{z} \leq \beta_G^{\polar}(z) \leq \chi_f(G,z)\) for
    every \(z \in \Reals_+^V\);
  \item\label{item:chi-to-alpha} if
    \(\norm[\infty]{z} \leq \beta_G(z) \leq \chi_f(G,z)\) for every
    \(z \in \Reals_+^V\), then
    \(\alpha(G,w) \leq \beta_G^{\polar}(w) \leq \norm[1]{w}\) for
    every \(w \in \Reals_+^V\);
  \item\label{item:bound-cauchy} \(\iprodt{w}{z} \leq \beta_G^{}(w)\beta_G^{\polar}(z)\) for
    every \(w,z \in \Reals_+^V\).
  \end{eqenum}
\end{theorem}
\begin{proof}
  \Cref{item:polar-bound-closed}:
  Immediate from \cref{item:polar-gauge-closed}.

  \Cref{item:bound-duality}:
  Immediate from \cref{item:gauge-duality}.

  \Cref{item:alpha-to-chi}:
  Note that \(\alpha(G,\cdot)\), \(\beta_G\), and (the restriction to
  \(\Reals_+^V\) of) \(\norm[1]{\cdot}\) are positive definite
  monotone gauges.
  Apply \cref{le:polar-gauge-reversal} to get
  \(\norm[\infty]{z} \leq \beta^{\polar}(G,z) \leq
  \alpha^{\polar}(G,z)\) for every \(z \in \Reals_+^V\).
  The rightmost term is \(\chi_f(G,z)\) by
  \cref{thm:alpha-chif-duality}.

  \Cref{item:chi-to-alpha}:
  Symmetric to the proof of~\cref{item:alpha-to-chi}.

  \Cref{item:bound-cauchy}:
  Immediate from \cref{item:gauge-cauchy}.
\end{proof}

\Cref{thm:bound-conversion} treats weighted bounds on~\(\alpha\)
and~\(\chi_f\) that are not necessarily graph parameters, in the sense
that the bounds might depend on vertex or edge labels.
We will see examples in \cref{sec:dual-Hoffman-bound}.
Let us become more precise about (weighted) graph parameters.
Let \(G = (V,E)\) be a graph, and let \(\sigma\) be a bijection
with domain~\(V\).
Denote by \(\sigma G\) the graph on vertex set \(\sigma(V)\) with
edges \(\setst{\sigma(i)\sigma(j)}{ij \in E}\).
If \(w \in \Reals_+^V\), denote \(\sigma w \in \Reals_+^{\sigma(V)}\)
defined by \(\paren{\sigma w}_{\sigma(i)} \coloneqq w_i\) for every
\(i \in V\).

Let \(\beta\) be a function that assigns a real number to each pair
\((G,w)\), where \(G = (V,E)\) is a graph and \(w \in \Reals_+^V\).
(We will not go into details about set-theoretic issues, e.g., we do
not discuss the \emph{class} of all graphs.)
We say that \(\beta\) is a \emph{(weighted) graph parameter} if,
whenever \(G = (V,E)\) is a graph and \(\sigma\) is a bijection with
domain~\(V\), we have \(\beta(\sigma G, \sigma w) = \beta(G,w)\) for
every \(w \in \Reals_+^V\).
That~is, graph parameters depend only on the isomorphism class of the
input graph.

Below and throughout the rest of the paper, we will deal with weighted
graph parameters \(\beta\) as in the previous paragraph.
For every graph \(G = (V,E)\),
\begin{equation}
  \label{eq:sub-G-abbrev}
  \text{%
    \emph{we may abbreviate
      \(\beta_G(\cdot) \coloneqq \beta(G,\cdot)\) without further
      mention}.
  }
\end{equation}
\begin{theorem}
  \label{thm:param-conversion}
  Let \(\beta\) be a weighted graph parameter such that \(\beta_G\) is
  a positive definite monotone gauge for every graph~\(G\).
  Define \(\beta^{\polar}\) by setting
  \(\beta^{\polar}(G,w) \coloneqq \beta_G^{\polar}(w)\) for every
  \(w \in \Reals_+^V\).
  Then:
  \begin{eqenum}[list:param-duality]
  \item\label{item:polar-param-closed} \(\beta^{\polar}\) is a
    weighted graph parameter;
  \item\label{item:param-duality} \(\beta^{\polar\polar} = \beta\);
  \item\label{item:param-cauchy}
    \(\iprodt{w}{z} \leq \beta_G^{}(w)\beta_G^{\polar}(z)\) for every
    \(w,z \in \Reals_+^V\), with equality whenever \(G\) is
    vertex-transitive and \(w = z = \ones\).
  \end{eqenum}
\end{theorem}
\begin{proof}
  \Cref{item:polar-param-closed}:
  Let \(G = (V,E)\) be a graph, and let \(\sigma \colon V \to U\) be a
  bijection.
  Let \(z \in \Reals_+^V\).
  Then
  \begin{equation*}
    \begin{split}
      \beta^{\polar}(\sigma G, \sigma z)
      =
      \beta_{\sigma G}^{\polar}(\sigma z)
      & =
      \max\setst{
        \iprodt{y}{(\sigma z)}
      }{
        y \in \Reals_+^U,\, \beta_{\sigma G}(y) \leq 1
      }
      \\
      & =
      \max\setst{
        \iprodt{(\sigma w)}{(\sigma z)}
      }{
        w \in \Reals_+^V,\, \beta_{\sigma G}(\sigma w) \leq 1
      }
      \\
      & =
      \max\setst{
        \iprodt{w}{z}
      }{
        w \in \Reals_+^V,\, \beta_G(w) \leq 1
      }
      =
      \beta_G^{\polar}(z)
      =
      \beta^{\polar}(G,z).
    \end{split}
  \end{equation*}

  \Cref{item:param-duality}:
  Immediate from
  \cref{item:bound-duality}.

  \Cref{item:param-cauchy}:
  The inequality is immediate from
  \cref{item:bound-cauchy}.
  Suppose \(G = (V,E)\) is vertex-transitive, and let \(\Aut(G)\)
  denote the automorphism group of~\(G\).
  Let
  \(\Ccal_{\beta,G} = \setst{x \in \Reals_+^V}{\beta_G(x) \leq 1}\) be
  the unit convex corner of~\(\beta_G\).
  Then
  \begin{equation}
    \label{eq:invariant-unit-convex-corner}
    \sigma \Ccal_{\beta,G}
    \coloneqq
    \setst{
      \sigma x
    }{
      x \in \Ccal_{\beta, G}
    }
    =
    \Ccal_{\beta,G},
    \qquad
    \forall \sigma \in \Aut(G).
  \end{equation}
  Indeed, if \(x \in \Ccal_{\beta,G}\) and \(\sigma \in \Aut(G)\),
  then
  \(\beta(G,\sigma x) = \beta(\sigma G, \sigma x) = \beta(G,x) \leq
  1\).
  This proves~`\(\subseteq\)'
  in~\cref{eq:invariant-unit-convex-corner}; for the reverse
  inclusion, apply the previous inclusion with \(\sigma^{-1}\) in
  place of~\(\sigma\).
  Let \(\bar{y}\) attain the maximum in the definition of
  \(\beta_G^{\polar}(\ones) = \max\setst{\iprodt{\ones}{y}}{y \in
    \Ccal_{\beta,G}}\).
  By \cref{eq:invariant-unit-convex-corner} and
  \cref{item:gauge-corner}, we apply the usual Reynolds operator to
  get that the point
  \begin{equation*}
    \tilde{y}
    \coloneqq
    \frac{1}{\card{\Aut(G)}}
    \sum_{\sigma \in \Aut(G)} \sigma \bar{y}
  \end{equation*}
  also attains the maximum, and it is constant on the orbits of the
  action from~\(\Aut(G)\) on~\(V\).
  Since~\(G\) is vertex-transitive, it follows that \(\tilde{y}\) is a
  scalar multiple of~\(\ones\).
  Say, \(\tilde{y} = \mu \ones \in \Ccal_{\beta,G}\) for some
  \(\mu \in \Reals_+\), so that \(\beta_G^{\polar}(\ones) = \mu n\).

  Analogously, by~\cref{item:polar-param-closed,item:param-duality},
  the \(\max\) in the definition of~\(\beta_G(\ones)\) is attained by
  a scalar multiple \(\eta \ones\) of~\(\ones\), for some
  \(\eta \in \Reals_+\), and \(\beta_G(\ones) = \eta n\).
  By \cref{item:polar-gauge-closed},
  the unit convex corner of~\(\beta_G^{\polar}\) is
  \(\abl(\Ccal_{\beta,G})\), so
  \(\eta \ones \in \abl(\Ccal_{\beta,G})\) whence
  \(\iprodt{(\mu\ones)}{(\eta\ones)} \leq 1\) so
  \(\mu\eta \leq n^{-1}\).
  Hence,
  \(\beta_G(\ones)\beta_G^{\polar}(\ones) = \mu\eta n^2 \leq n\), as
  desired.
\end{proof}

\section{The Dual of Hoffman's Lower Bound for the Chromatic Number}
\label{sec:dual-Hoffman-bound}

This section addresses the question from the introduction on a precise
duality relation between the Delsarte-Hoffman ratio
bound~\cref{eq:ratio-bound} and the Hoffman
bound~\cref{eq:Hoffman-bound}, using the duality theory of positive
definite monotone gauges from \cref{sec:gauge-duality}.
As described in that section, we first need to introduce a weighted
version of the lower bound~\cref{eq:Hoffman-bound}.

Let \(\Sym{V}\) be the set of symmetric \(V \times V\) matrices.
Let \(\Psd{V} \subseteq \Sym{V}\) denote the set of positive
semidefinite matrices, i.e., matrices whose eigenvalues are
nonnegative.
For matrices \(X,Y \in \Sym{V}\), write \(X \succeq Y\) if
\(X-Y \in \Psd{V}\).
Recall that each positive semidefinite matrix~\(W\) has a unique
positive semidefinite square root, denoted by \(W^{\half}\).
The linear map \(\Diag \colon \Reals^V \to \Sym{V}\) builds diagonal
matrices, that~is, if \(x \in \Reals^V\), then
\(\Diag(x) \in \Sym{V}\) is a diagonal matrix with
\([\Diag(x)]_{ii} = x_i\) for every \(i \in V\).
The linear map \(\diag \colon \Sym{V} \to \Reals^V\) extracts the
diagonal of a (symmetric) matrix.

It turns out that the proof of (a strengthening
of)~\cref{eq:Hoffman-bound} works even when the adjacency
matrix~\(A_G\) of our graph~\(G = (V,E)\) is replaced with any
\emph{generalized adjacency matrix} of~\(G\), i.e., any matrix in
\begin{equation*}
  \Acal_G
  \coloneqq
  \setst{
    A \in \Sym{V}
  }{
    A_{ij} = 0 \text{ if } ij \notin E
  }.
\end{equation*}
That~is, the proof of~\cref{eq:Hoffman-bound} relies only on the fact
that the matrix has zero diagonal and zeroes on the off-diagonal
entries corresponding to non-edges.
For any nonzero matrix \(A \in \Sym{V}\) with zero diagonal, denote
\begin{equation}
  \label{eq:At-def}
  \At \coloneqq \frac{A}{-\lambda_{\min}(A)},
\end{equation}
and if \(A = 0\), define \(\At \coloneqq A\).
Note that \(\At \succeq -I\), so
\begin{equation}
  \label{eq:I+At-key-properties}
  \text{
    \(I+\At \succeq 0\),
    and
    \(\Null(I+\At)\)
    is the \(\lambda_{\min}(A)\)-eigenspace of~\(A\) if \(A \neq 0\).
  }
\end{equation}

We can now define the weighted Hoffman bound: for every
\(A \in \Sym{V}\) with zero diagonal and every \(w \in \Reals_+^V\),
define
\begin{equation}
  \label{eq:def-Hoffman-A}
  \Hoffman(A,w)
  \coloneqq
  \lambda_{\max}\paren[\Big]{
    W^{\half}(I + \At)W^{\half}
  }
  \qquad
  \text{where }
  W \coloneqq \Diag(w),
\end{equation}
and for a graph \(G = (V,E)\), define
\begin{equation*}
  \Hoffman(G,w)
  \coloneqq
  \Hoffman(A_G,w).
\end{equation*}
Unless \(G\) has no edges,
\begin{equation*}
  \Hoffman(G,\ones)
  =
  \lambda_{\max}\paren[\Big]{I + \frac{A_G}{-\lambda_{\min}(A_G)}}
  =
  1 - \frac{\lambda_{\max}(A_G)}{\lambda_{\min}(A_G)},
\end{equation*}
which is precisely the quantity in the RHS of \cref{eq:Hoffman-bound}.
As in \cref{eq:sub-G-abbrev}, for the function
\(\beta \coloneqq \Hoffman\) and later for other functions~\(\beta\)
of a matrix \(A \in \Sym{V}\) and weights \(w \in \Reals_+^E\),
\begin{equation}
  \text{%
    \emph{we may abbreviate
      \(\beta_A(\cdot) \coloneqq \beta(A,\cdot)\) without further
      mention}.
  }
\end{equation}

Let us verify that the bound~\cref{eq:def-Hoffman-A} fits into
our framework from \cref{sec:gauge-duality}:
\begin{theorem}
  \label{thm:Hoffman-pdmg-bound}
  Let \(G = (V,E)\) be a graph.
  For every \(A \in \Acal_G\), the function
  \(\Hoffman_A \colon \Reals_+^V \to \Reals\) is a positive definite
  monotone gauge and, for every \(w \in \Reals_+^V\),
  \begin{equation}
    \label{eq:Hoffman-bounds-chif}
    \norm[\infty]{w} \leq \Hoffman(A,w) \leq \chi_f(G,w).
  \end{equation}
\end{theorem}
The fact~\cref{eq:Hoffman-bounds-chif} follows from known relations
between \(\Hoffman(A,\cdot)\) and the Lovász theta function (see,
e.g., \cite[Theorem~6]{Lovasz79a} or \cite[Theorem~33]{Knuth94a}), and
a weighted version of~\cref{eq:sandwich-long}; we include below a
proof that bypasses the theta function.
We will use the following well-known fact in the proof of
\cref{thm:Hoffman-pdmg-bound} and elsewhere, extensively.
\begin{lemma}[{see \cite[Theorem 1.3.20]{HornJ90a}}]
  \label{lem:eigen-swap}
  Let \(U,V\) be finite sets, and let \(A \in \Reals^{U\times V}\) and
  \(B \in \Reals^{V \times U}\).
  Then \(AB\) and \(BA\) have the same nonzero eigenvalues (taking
  multiplicities into account).
\end{lemma}

\begin{corollary}
  \label{cor:Hoffman-swapped}
  Let \(G = (V,E)\) be a graph.
  For every \(A \in \Acal_G\),
  \begin{equation*}
    \Hoffman(A,w)
    =
    \lambda_{\max}\paren[\Big]{
      (I+\At)^{\half}
      \Diag(w)
      (I+\At)^{\half}
    }.
  \end{equation*}
\end{corollary}
\begin{proof}
  Immediate from~\cref{eq:I+At-key-properties,lem:eigen-swap}.
\end{proof}

Now we are ready to prove~\cref{thm:Hoffman-pdmg-bound}.

\begin{proof}[Proof of \cref{thm:Hoffman-pdmg-bound}]
  Let \(A \in \Acal_G\).
  Clearly, \(\Hoffman_A(0) = 0\) and \(\Hoffman_A\) is positively
  homogeneous.
  Let \(w \in \Reals_+^V\) and set \(W \coloneqq \Diag(w)\).
  If \(i \in V\) and \(e_i \in \Reals^V\) denotes the \(i\)th standard
  basis vector, then
  \[\Hoffman_A(w) = \max_{\norm[2]{h}=1}
    \qform{W^{\half}(I+\At)W^{\half}}{h} \geq
    \qform{W^{\half}(I+\At)W^{\half}}{e_i} = w_i (I+\At)_{ii} = w_i.\]
  Thus,
  \begin{equation}
    \label{eq:Hoffman-vs-inftynorm}
    \norm[\infty]{w} \leq \Hoffman_A(w).
  \end{equation}
  In~particular, \(\Hoffman_A\) is nonnegative.
  Since
  \(\lambda_{\max}(X) = \max\setst{\qform{X}{h}}{h \in
    \Reals^V,\,\norm[2]{h} = 1}\) for every \(X \in \Sym{V}\), we have
  that \(\lambda_{\max}(X)\) is the pointwise supremum of linear
  functions of~\(X\), and thus it is convex.
  Now \cref{cor:Hoffman-swapped} shows that \(\Hoffman_A\) is the
  composition of a convex function after a linear function, whence
  \(\Hoffman_A\) is convex.
  Thus, \(\Hoffman_A\) is a gauge; recall that, under the assumption
  of positive homogeneity, sublinearity and convexity are equivalent.
  If \(z \in \Reals_+^V\) is such that \(w \leq z\), then
  \(W = \Diag(w) \preceq \Diag(z) \eqqcolon Z\) so
  \((I+\At)^{\half} W (I+\At)^{\half} \preceq (I+\At)^{\half} Z
  (I+\At)^{\half}\).
  Together with this fact, \cref{cor:Hoffman-swapped} shows that,
  \(\Hoffman_A(w) \leq \Hoffman_A(z)\), so \(\Hoffman_A\) is monotone.
  Now \cref{eq:Hoffman-vs-inftynorm} concludes the proof that
  \(\Hoffman_A\) is a positive definite monotone gauge.

  Next we prove the second inequality
  in~\cref{eq:Hoffman-bounds-chif}.
  Again, let \(w \in \Reals_+^V\) and set \(W \coloneqq \Diag(w)\).
  We may assume that \(w \neq 0\).
  By moving to an induced subgraph of~\(G\) if appropriate, we may
  assume that \(w > 0\).
  Abbreviate
  \(\Lambda \coloneqq \lambda_{\max}\paren{W^{\half}(I+\At)W^{\half}}
  = \Hoffman_A(w)\).
  Let \(b \in \Reals^V\) be a unit-norm \(\Lambda\)-eigenvector of
  \(W^{\half}(I+\At)W^{\half}\), and set \(B \coloneqq \Diag(b)\).
  Set \(X \coloneqq W^{-\half}B(I+\At)BW^{-\half}\).
  We claim that
  \begin{equation}
    \label{eq:Hoffman-bounds-chif-aux0}
    \iprodt{
      \incidvector{S}
    }{
      \paren[\big]{
        2Xw - \Lambda W^{-1}Bb
      }
    }
    \leq 1,
    \qquad
    \text{whenever \(S \subseteq V\) is stable}.
  \end{equation}
  Since \(I+\At\) is positive semidefinite, so is~\(X\).
  Hence,
  \begin{equation}
    \label{eq:Hoffman-bounds-chif-aux1}
    0
    \leq
    \qform{
      X
    }{
      \paren{
        \Lambda^{\half}\incidvector{S}
        -
        \Lambda^{-\half}w
      }
    }
    =
    \Lambda\qform{X}{\incidvector{S}}
    -
    2\iprodt{\incidvector{S}}{Xw}
    +
    \Lambda^{-1}\qform{X}{w}
  \end{equation}
  The rightmost term in the RHS of~\cref{eq:Hoffman-bounds-chif-aux1}
  is
  \begin{equation}
    \label{eq:Hoffman-bounds-chif-aux2}
    \Lambda^{-1}\qform{X}{w}
    =
    \Lambda^{-1} \qform{WW^{-\half}B{(I+\At)}BW^{-\half}W}{\ones}
    =
    \Lambda^{-1} \qform{BW^{\half}{(I+\At)}W^{\half}B}{\ones}
    =
    1,
  \end{equation}
  since diagonal matrices commute.
  The leftmost term in the RHS of~\cref{eq:Hoffman-bounds-chif-aux1}
  is
  \begin{equation}
    \label{eq:Hoffman-bounds-chif-aux3}
    \Lambda \qform{X}{\incidvector{S}}
    =
    \Lambda \qform{W^{-\half}B^2 W^{-\half}}{\incidvector{S}}
    =
    \Lambda \qform{\Diag(W^{-1}Bb)}{\incidvector{S}}
    =
    \Lambda \iprodt{\incidvector{S}}{W^{-1}Bb},
  \end{equation}
  where the first equation uses the facts that \(S\) is stable and
  \(A \in \Acal_G\).
  The proof of \cref{eq:Hoffman-bounds-chif-aux0} follows by
  combining~\cref{eq:Hoffman-bounds-chif-aux1,%
    eq:Hoffman-bounds-chif-aux2,eq:Hoffman-bounds-chif-aux3}.

  Let \(y\) be an optimal solution for the LP on the RHS
  of~\cref{eq:def-chif}.
  We may assume that
  \(\sum_{S \in \Scal(G)} y_S \incidvector{S} = w\).
  Now we multiply each inequality in~\cref{eq:Hoffman-bounds-chif-aux0} by \(y_S \geq 0\)
  and sum them all together to get
  \begin{equation}
    \label{eq:Hoffman-bounds-chif-aux4}
    \iprodt{w}(2Xw - \Lambda W^{-1}Bb) \leq \iprodt{\ones}{y}.
  \end{equation}
  The leftmost term is \(2\qform{X}{w} = 2\Lambda\)
  by~\cref{eq:Hoffman-bounds-chif-aux2}, and the second term in the
  LHS is
  \(\Lambda \iprodt{\ones}{WW^{-1}Bb} = \Lambda \iprodt{b}{b} =
  \Lambda\).
  Hence, \cref{eq:Hoffman-bounds-chif-aux4} yields
  \(\Hoffman_A(w) = \Lambda \leq \iprodt{\ones}{y} = \chi_f(G,w)\), as
  desired.
\end{proof}

We now define a new bound for the weighted stability number via an
SDP, which we will prove to be the (gauge) dual of \(\Hoffman_A\).
Let \(A \in \Sym{V}\) have zero diagonal.
For every \(w \in \Reals_+^V\), define
\begin{equation}
  \label{eq:def-Hdual-A}
  \Hdual(A,w)
  \coloneqq
  \max\setst{
    \iprodt{w}{x}
  }{
    x \in \Reals_+^V,\,
    \paren{I + \At}^{\half} \Diag(x) \paren{I + \At}^{\half} \preceq I
  }.
\end{equation}
As before, for a graph \(G\) define \(\Hdual(G,w) \coloneqq
\Hdual(A_G,w)\).
Note that the semidefinite constraint in~\cref{eq:def-Hdual-A} may be
rewritten as \(\sum_{i \in V} x_i^{} \oprodsym{b_i} \preceq I\) where
\(b_i\) is the \(i\)th column of \(\paren{I+\At}^{\half}\).
Hence, that constraint is a linear matrix inequality (LMI), and this
optimization problem is an SDP.
The dual SDP is
\begin{equation}
  \label{eq:Hdual-dual-SDP-pre}
  \min\setst[\big]{
    \trace(Y)
  }{
    Y \in \Psd{V},\,
    \diag\paren[\big]{\paren{I + \At}^{\half}Y\paren{I + \At}^{\half}}
    \geq w
  },
\end{equation}
where \(\trace\) is the trace.
The feasible point \({x = 0}\) is a relaxed Slater point of the
SDP~\cref{eq:def-Hdual-A}, that~is, a feasible solution where the
slack \(I - {\paren{I+\At}^{\half}\Diag(0)\paren{I+\At}^{\half}}\)
corresponding to the LMI constraint is positive definite.
Similarly, \(\paren{\norm[\infty]{w}+1}I\) is a relaxed Slater
point of the SDP~\cref{eq:Hdual-dual-SDP-pre} for every
\(w \in \Reals^V_+\).
Hence, by SDP Strong Duality (see, e.g.,
\cite[Theorem~7.1.2]{Nemirovski12a}), both primal and dual SDPs have
optimal solutions (which justifies our use of `\(\max\)'
and~`\(\min\)' above), and their optimal values coincide.
Hence,
\begin{equation}
  \label{eq:Hdual-dual-SDP}
  \Hdual(A,w)
  =
  \min\setst[\big]{
    \trace(Y)
  }{
    Y \in \Psd{V},\,
    \diag\paren[\big]{\paren{I + \At}^{\half}Y\paren{I + \At}^{\half}}
    \geq w
  }.
\end{equation}

The carefully crafted framework from \cref{sec:gauge-duality} now pays
off by providing a sleek proof of the main duality results of this
section:
\begin{theorem}
  \label{thm:Hdual-is-Hoffman-dual}
  Let \(G = (V,E)\) be a graph, and let \(A \in \Acal_G\).
  Then \(\Hdual_A^{} = \Hoffman_A^{\polar}\).
\end{theorem}
\begin{proof}
  Let \(x \in \Reals_+^V\).
  \Cref{cor:Hoffman-swapped} implies that \(\Hoffman(A,x) \leq 1\)
  holds if and only if
  \((I + \At)^{\half}\Diag(x)(I + \At)^{\half} \preceq I\).
  Hence, for every \(w \in \Reals_+^V\),
  \begin{equation*}
    \begin{split}
      \Hoffman^{\polar}(A,w)
      &= \max\setst{
        \iprodt{w}{x}
      }{
        x \in \Reals_+^V,\; \Hoffman(A,x) \leq 1
      }\\
      &= \max\setst[\big]{
        \iprodt{w}{x}
      }{
        x \in \Reals_+^V,\;
        (I + \At)^{\half}\Diag(x)(I + \At)^{\half} \preceq I
      }\\
      &= \Hdual(A,w).\qedhere
    \end{split}
  \end{equation*}
\end{proof}

\begin{corollary}
  \label{cor:Hdual-pdmg-bound}
  Let \(G = (V,E)\) be a graph.  For every \(A \in \Acal_G\), the
  function \(\Hdual_A \colon \Reals_+^V \to \Reals\) is a positive
  definite monotone gauge and, for every \(w \in \Reals_+^V\),
  \begin{equation}
    \label{eq:Hdual-bounds-alpha}
    \alpha(G,w) \leq \Hdual(A,w) \leq \norm[1]{w}.
  \end{equation}
\end{corollary}
\begin{proof}
  Immediate from \cref{thm:Hoffman-pdmg-bound,thm:bound-conversion,%
    thm:Hdual-is-Hoffman-dual}.
\end{proof}

It is instructive to see how \cref{eq:Hdual-bounds-alpha} can be
derived directly, without any mention to the Hoffman
bound~\(\Hoffman_A\). Using the facts that both \(\Hoffman_A\) and
\(\Hdual_A\) are positive definite monotone gauges and that they are
gauge dual to each other, and using \cref{thm:bound-conversion}, the
proof of \cref{eq:Hdual-bounds-alpha} below provides an alternative
proof of \cref{eq:Hoffman-bounds-chif}.

Note that the feasible region
\begin{equation*}
  \Ucal_A
  \coloneqq
  \setst{
    x \in \Reals_+^V
  }{
    (I+\At)^{\half}\Diag(x)(I+\At)^{\half} \preceq I
  }
\end{equation*}
of the primal SDP~\cref{eq:def-Hdual-A} is easily checked to be a
convex corner.
Hence, \(\Hdual_A\) is a positive definite monotone gauge by
\cref{item:gauge-construction-primal}.
Since both optimization problems defining \(\alpha_G\) and
\(\Hdual_A\) maximize the same objective function, we can prove that
\(\alpha(G,w) \leq \Hdual(A,w)\) for every \(w \in \Reals_+^V\) by
showing that
\begin{equation}
  \label{eq:STAB-in-U-A}
  \STAB(G)
  \subseteq
  \Ucal_A.
\end{equation}
Since \(\Ucal_A\) is convex, it suffices to prove that
\(\incidvector{S} \in \Ucal_A\) for every stable set
\(S \subseteq V\).
But this follows from a simple calculation: if \(S \subseteq V\) is
stable, then \(\incidvector{S} \in \Ucal_A\) is equivalent to
\(\Diag(\incidvector{S})^{\half} (I + \At)
\Diag(\incidvector{S})^{\half} \preceq I\) by \cref{lem:eigen-swap},
and the LHS is
\(\Diag(\incidvector{S}) (I+\At) \Diag(\incidvector{S}) =
\Diag(\incidvector{S})\), which clearly satisfies
\(\Diag(\incidvector{S}) \preceq I\).
This proves\cref{eq:STAB-in-U-A}.

Let \(x \in \Ucal_A\).
\cref{lem:eigen-swap} implies
\({\Diag(x)^{\half}(I + \At)\Diag(x)^{\half}} \preceq I\).
Comparing diagonals, we reach \(x \le \ones\).
In other words, the set \(\Ucal_A\) is a subset of
\(\setst{x \in \Reals_+^V}{\norm[\infty]{x} \le 1}\), which implies
that \(\Hdual(A,w) \le \norm[1]{w}\) for every \(w \in \Reals_+^V\)
and concludes our alternative proof of \cref{eq:Hdual-bounds-alpha}.

At this point, we have almost fully answered the question which
introduces this text.
It remains only to show that our new bound reduces into the well-known
ratio bound~\cref{eq:ratio-bound} for regular graphs.
We start with a case slightly more general than that of regular
graphs.

We shall make use of the \emph{Moore–Penrose pseudoinverse}
\(M^{\MPinv}\) of a symmetric matrix \(M \in \Sym{V}\).
We rely on very few properties of~\(M^{\MPinv}\), which we include
here (see \cite{Ben-IsraelG03a} for further properties).
If \(Mx = \lambda x\) for some nonzero \(x \in \Reals^V\) and nonzero
\(\lambda \in \Reals\), i.e., if \(x \in \Reals^V\) is an eigenvector
of~\(M\) with nonzero eigenvalue \(\lambda\), then
\(M^{\MPinv}x = \lambda^{-1}x\).
The pseudoinverse commutes with positive semidefinite square roots,
that~is, \((M^{\half})^{\MPinv} = (M^{\MPinv})^{\half}\), and we use
the common shorthand notation \(M^{\MPhalf}\) for both of them.
Finally, \(M M^{\MPinv}\) is the orthogonal projection onto the range
of~\(M\).

\begin{theorem}
  \label{thm:Hdual-yields-ratio-bound-pre}
  Let \(G\) be a graph and let \(A \in \Acal_G\) be nonzero such that
  \(A \ones = \lambda\ones\) for \(\lambda \coloneqq
  \lambda_{\max}(A)\).
  Denote \(\tau \coloneqq \lambda_{\min}(A)\).
  Then
  \[\Hdual(A,\ones) = \frac{n}{1 - \lambda/\tau}.\]
\end{theorem}
\begin{proof}
  We have
  \begin{equation*}
    (I + \At)\ones
    = \paren[\Big]{1 - \frac{\lambda}{\tau}} \ones,
    \qquad
    \text{and}
    \qquad
    (I + \At)^{\MPinv}\ones
    = \paren[\Big]{1 - \frac{\lambda}{\tau}}^{-1} \ones,
  \end{equation*}

  To find the optimal value of the SDP~\cref{eq:def-Hdual-A}, it
  suffices by SDP weak duality to exhibit a pair of primal and dual
  feasible solutions with the same objective value.
  We start with the primal SDP~\cref{eq:def-Hdual-A}.
  Define \(x \coloneqq \mu \ones \geq 0\), where
  \(\mu \coloneqq \paren{1 - \lambda/\tau}^{-1}\).
  Note that \(x\) is feasible in~\cref{eq:def-Hdual-A}, since
  \(\lambda_{\max}\paren[\big]{ (I + \At)^{\half} \Diag(x) (I +
    \At)^{\half} } = \mu \lambda_{\max}(I + \At) = 1\), and its
  objective value is \(\iprodt{\ones}{x}
  = n/\paren{1 - \lambda/\tau}\).

  A dual feasible solution with the same value is
  \(Y \coloneqq (I + \At)^{\MPhalf} \ones\ones^\transp (I +
  \At)^{\MPhalf} \succeq 0\).
  Since \(\ones\) is an eigenvector of \((I + \At)^{\half}\), it
  follows that
  \(\ones = (I + \At)^{\half}(I + \At)^{\MPhalf}\ones\), which
  implies
  \(\diag\paren[\big]{(I + \At)^{\half}Y(I + \At)^{\half}} =
  \diag(\ones\ones^\transp) = \ones\).
  Thus, \(Y\) is feasible in~\cref{eq:Hdual-dual-SDP}, and its
  objective value is
  \begin{equation*}
    \trace(Y)
    = \trace\paren[\Big]{
      (I + \At)^{\MPhalf}
        \ones\ones^\transp
      (I + \At)^{\MPhalf}
    }
    = \ones^\transp (I + \At)^{\MPinv} \ones
    = \paren{1 - \lambda/\tau}^{-1} \ones^\transp\ones
    = \frac{n}{1 - \lambda/\tau} = \iprodt{\ones}{x}. \qedhere
  \end{equation*}
\end{proof}

The next corollary is an immediate consequence of
\cref{thm:Hdual-yields-ratio-bound-pre} when applied to the adjacency
matrix~\(A_G\) of a regular graph~\(G\), and it proves that the
Delsarte-Hoffman ratio bound~\cref{eq:ratio-bound} and the Hoffman
bound~\cref{eq:Hoffman-bound} are dual to each other, since
\(\Hdual_G\) and \(\Hoffman_G\) are weighted graph \emph{parameters}
dual to each other in the sense of \cref{thm:param-conversion} by
\cref{thm:Hoffman-pdmg-bound,thm:Hdual-is-Hoffman-dual}.
\begin{corollary}
  \label{cor:Hdual-yields-ratio-bound}
  Let \(G\) be a \(k\)-regular graph, with \(k \geq 1\).
  Denote \(\tau \coloneqq \lambda_{\min}(A_G)\).
  Then
  \[\Hdual(G,\ones) = \frac{n}{1 - k/\tau}.\]
\end{corollary}

The proof of \cref{thm:Hdual-yields-ratio-bound-pre} above suggests
using a scalar multiple of a Perron-Frobenius eigenvector of a
connected graph (see \cite[Theorem~8.8.1]{GodsilR01a}) to get a
feasible solution for the dual SDP~\cref{eq:Hdual-dual-SDP}, which
yields another extension of the ratio bound~\cref{eq:ratio-bound}:
\begin{proposition}
  \label{prop:alpha-bound-Perron}
  Let \(G = (V,E)\) be a connected graph.
  Set \(\lambda \coloneqq \lambda_{\max}(A_G)\) and
  \(\tau \coloneqq \lambda_{\min}(A_G)\).
  Let \(p \in \Reals_+^V\) be the unit-norm Perron-Frobenius
  eigenvector of \(A_G\).
  Then
  \begin{equation}
    \label{eq:alpha-bound-Perron}
    \alpha(G)
    \leq
    \frac{
      \max_{i \in V} p_i^{-2}
    }{
      1 - \lambda/\tau
    }.
  \end{equation}
\end{proposition}
\begin{proof}
  Let \(\eta\) denote the RHS in~\cref{eq:alpha-bound-Perron}.
  Then
  \(\trace\paren[\big]{\eta pp^\transp} = \eta \norm[2]{p}^2 = \eta\).
  Moreover,
  \begin{equation*}
    \diag\paren[\big]{
      (I + \At_G)^{\half}\eta pp^\transp(I + \At_G)^{\half}
    }
    =
    \eta \paren[\big]{1 - k/\tau} \diag(pp^\transp)
    \geq
    \ones.
  \end{equation*}
  Hence \(\eta pp^\transp\) is feasible in the dual
  SDP~\cref{eq:Hdual-dual-SDP} with objective value~\(\eta\), so
  \(\alpha(G) \leq \Hdual(A_G,\ones) \leq \eta\) by
  \cref{eq:Hdual-bounds-alpha}.
\end{proof}

When \cref{prop:alpha-bound-Perron} is applied to a \emph{connected}
regular graph, we recover \cref{cor:Hdual-yields-ratio-bound}.
In fact, we point out that Proposition~\ref{prop:alpha-bound-Perron}
can be obtained elementarily.
We shall use the same notation for \(\lambda\), \(\tau\), \(p\), and
\(\eta\) from \cref{prop:alpha-bound-Perron} and its proof.
Note that \(I+\At_G \succeq (1-\lambda/\tau)\oprodsym{p}\).
If \(S\) is a maximum stable set, then
\[
  \alpha(G)
  =
  \qform{(I + \At_G)}{\incidvector{S}}
  \geq
  \paren[\Big]{
    1 - \frac{\lambda}{\tau}
  }
  \iprodt{\incidvector{S}}{p}
  \iprodt{p}{\incidvector{S}}
  \geq
  \frac{1}{\eta}
  \iprodt{\incidvector{S}}{\ones}
  \iprodt{\ones}{\incidvector{S}}
  =
  \frac{\alpha(G)^2}{\eta}.
\]

By \cref{thm:Hoffman-pdmg-bound} and the results from
\cref{sec:gauge-duality} (see, e.g., the paragraph that follows
\cref{thm:gauge-duality}), the Hoffman bound \(\Hoffman_A\) may be
expressed as \emph{linear} optimization of the function
\(x \in \Reals^V \mapsto \iprodt{w}{x}\) with \(x\) ranging over some
convex set.
We provide explicit descriptions of such sets below.
They will be used in \cref{sec:theta} to provide new descriptions of
the theta body, which is the convex corner over which the Lovász theta
function optimizes.  We will make use of the following notation: for a
graph \(G = (V,E)\) and \(A \in \Acal_G\), set
\begin{equation}
  \label{eq:def-H-A}
  \Hcal_A
  \coloneqq
  \setst[\big]{
    x \in \Reals_+^V
  }{
    \exists Y \in \Psd{V},\,
    \trace(Y) \leq 1,\,
    \diag\paren[\big]{(I+\At)^{\half}Y(I+\At)^{\half}} \geq x
  }.
\end{equation}
\begin{theorem}
  \label{thm:Hoffman-max-convex}
  Let \(G = (V,E)\) be a graph, and let \(A \in \Acal_G\).
  Then, for every \(w \in \Reals_+^V\),
  \begin{align}
    \label{eq:Hoffman-over-convex-set}
    \Hoffman(A,w)
    & =
    \max\setst[\big]{
      \iprodt{w}{\diag\paren[\big]{(I+\At)^{\half}X(I+\At)^{\half}}}
    }{
      X \in \Psd{V},\,
      \trace(X) = 1
    }
    \\
    \label{eq:Hoffman-over-convex-corner}
    & =
    \max\setst[\big]{
      \iprodt{w}{x}
    }{
      x \in \Hcal_A
    }.
  \end{align}
  and the feasible region~\(\Hcal_A\) of the second maximization
  problem is the unit convex corner of~\(\Hdual(A,\cdot)\).
\end{theorem}
\begin{proof}
  By \cref{cor:Hoffman-swapped}, we can formulate \(\Hoffman(A,w)\) as
  the SDP
  \begin{equation*}
    \Hoffman(A,w)
    =
    \min\setst[\big]{
      \mu \in \Reals
    }{
      \paren{I+\At}^{\half} \Diag(w) \paren{I+\At}^{\half} \preceq \mu I
    },
  \end{equation*}
  which clearly has a relaxed Slater point and is bounded below.
  Hence, by SDP Strong Duality, the dual SDP has an optimal solution
  and the same optimal value as the primal SDP:
  \begin{equation}
    \label{eq:Hoffman-dual-SDP-swapped}
    \Hoffman(A,w)
    =
    \max\setst{
      \trace\paren{
        \paren{I + \At}^{\half} \Diag(w) \paren{I + \At}^{\half} X
      }
    }{
      X \in \Psd{V},\,
      \trace(X) = 1
    }.
  \end{equation}
  Since the objective function in \cref{eq:Hoffman-dual-SDP-swapped}
  may be rewritten as
  \(X \in \Sym{V} \mapsto \iprodt{w}{\diag(\paren{I + \At}^{\half} X
    \paren{I + \At}^{\half})}\), the proof of
  \cref{eq:Hoffman-over-convex-set} is complete.
  To prove \cref{eq:Hoffman-over-convex-corner}, by
  \cref{thm:Hoffman-pdmg-bound,%
    thm:Hdual-is-Hoffman-dual,item:gauge-duality}, it suffices to
  prove that
  \begin{equation*}
    \Hcal_A = \setst{x \in \Reals_+^V}{\Hdual(A,x) \leq 1}.
  \end{equation*}
  However, this is immediate from the dual
  formulation~\cref{eq:Hdual-dual-SDP}.
\end{proof}

Note that the feasible region in \cref{eq:Hoffman-over-convex-set} is
not a convex corner like the one in
\cref{eq:Hoffman-over-convex-corner}, however the latter feasible
region involves a projection (of~\(Y\) in \cref{eq:def-H-A}) whereas
the former is projection-free.

The definitions \(\Hoffman(G,w) = \Hoffman(A_G,w)\) and
\(\Hdual(G,w) = \Hdual(A_G,w)\) create graph parameters.
There is, however, a more interesting approach.
Consider \(\Hoffman\) and \(\Hdual\) as functions defined for every
pair \((A,w)\) where \(w \in \Reals_+^V\) and \(A \in \Sym{V}\) is
such that \(\diag(A) = 0\).
For a given graph \(G\) and \(w \in \Reals^V\), the set \(\Acal_G\)
defines many bounds, and we can simply choose the best one.
In other words, to find the best lower bound for \(\chi_f(G,w)\),
consider
\[
  \sup_{A \in \Acal_G} \Hoffman(A,w),
\]
and to find the best upper bound for \(\alpha(G,w)\) consider
\[
 \inf_{A \in \Acal_G} \Hdual(A,w).
\]
The expressions above define functions which actually depend on
\((G,w)\), and it is possible to prove them to be graph parameters,
i.e., to be constant on isomorphism classes of graphs.
There is, however, no need to do so, since \cref{sec:theta} will show
that both graph parameters just mentioned are actually well known.

\section{Relation with Luz's Convex Quadratic Programming Bound}
\label{sec:Luz}

Luz~\cite{Luz95a} introduced a convex quadratic program (CQP) that
bounds the stability number, which was later generalized to the
weighted case in~\cite{LuzC01a,CarliT17a}; we will use the weighted
version from~\cite{CarliT17a}.
Let \(G = (V,E)\) be a graph, let \(A \in \Acal_G\), and set \(\At\)
as in \cref{eq:At-def}.
Denote the componentwise square root of a nonnegative vector
\(w \in \Reals_+^V\) as \(\sqrt{w} \in \Reals_+^V\).
For every \(x \in \Reals^V\), define the orthogonal projector
\(\projsupp x \coloneqq \sum \setst{\oprodsym{e_i}}{i \in V,\, x_i
  \neq 0} \in \Sym{V}\).
For every \(w \in \Reals_+^V\), define \(\Luz(A,w)\) as the optimal
value of the following CQP:
\begin{align}
  \label{eq:def-Luz}
  \Luz(A,w)
  & \coloneqq \sup\setst[\big]{
  2\iprodt{w}{x}
  -
  \iprodt{x}{W^{\half}(I + \At)W^{\half}x}
  }{
  x \in \Reals_+^V
  },
  \qquad
  \text{where }
  W \coloneqq \Diag(w).
  \\
  \label{eq:def-Luz-alt-supp}
  & = \sup\setst[\big]{
    2\iprodt{\sqrt{w}}{x}
    -
    \qform{(I+\At)}{x}
  }{
    x \in \Reals_+^V,\,
    \projsupp w x = x
  }.
\end{align}
We write `\(\sup\)' rather than `\(\max\)' because \(\Luz(A,w)\) may
be \(+\infty\); we will prove this below.
(One may use the changes of variables \(x \mapsto W^{\half}x\) and
\(x \mapsto W^{\MPhalf}x\) to prove equivalence between
formulations~\cref{eq:def-Luz,eq:def-Luz-alt-supp}, as well as the
fact that \(W^{\half} W^{\MPhalf} = \projsupp w\).)
To see that
\begin{equation}
  \label{eq:alpha-leq-Luz}
  \alpha(G,w) \leq \Luz(A,w),
\end{equation}
note that the objective value of \(x \coloneqq \incidvector{S}\) is
\(\iprodt{w}{x}\) in~\cref{eq:def-Luz} whenever \(S \subseteq V\) is
stable.
In fact, in \cite[Corollary~29]{CarliT17a} it is proved that
\begin{equation}
  \label{eq:theta-best-Luz-pre}
  \theta(G,w) = \min_{A \in \Acal_G} \Luz(A,w),
\end{equation}
extending the unweighted version first proved by Luz and
Schrijver~\cite{LuzS05a}; see also~\cite{Luz16a}.

In this section, we study two results about the optimization problem
\cref{eq:def-Luz}.
First, we show that for every \(A \in \Acal_G\) the new upper bound
\(\Hdual_A\) is bounded above by \(\Luz_A\).
We then proceed to show that for every \emph{nonnegative} generalized
adjacency matrix of \(G\), the parameters \(\Hdual_A\) and \(\Luz_A\)
actually coincide.

The (Lagrangean) dual of the CQP~\cref{eq:def-Luz-alt-supp} can be
formulated as
\begin{equation*}
  \inf\setst[\big]{
    \qform{(I+\At)}{y}
  }{
    y \in \Reals^V,\,
    \projsupp w (I+\At)y \geq \sqrt{w}\,
  },
\end{equation*}
which is equivalent to
\begin{equation}
  \label{eq:Luz-dual-pre-supp}
  \inf\setst[\big]{
    \norm[2]{y}^2
  }{
    y \in \Reals^V,\,
    \projsupp w (I+\At)^{\half} y \geq \sqrt{w}\,
  }.
\end{equation}
By Convex Optimization Strong Duality (see, e.g., \cite{BoydV04a}) the
optimal values of~\cref{eq:def-Luz-alt-supp,eq:Luz-dual-pre-supp}
coincide and~\cref{eq:Luz-dual-pre-supp} has an optimal solution
whenever it has a feasible solution:
\begin{equation}
  \label{eq:Luz-dual-supp}
  \Luz(A,w)
  =
  \min\setst[\big]{
    \norm[2]{y}^2
  }{
    y \in \Reals^V,\,
    \projsupp w (I+\At)^{\half} y \geq \sqrt{w}\,
  }
  \qquad
  \text{whenever }
  \Luz(A,w) < \infty.
\end{equation}

\begin{theorem}
  \label{thm:Hdual-leq-Luz}
  Let \(G = (V,E)\) be a graph, let \(A \in \Acal_G\), and let
  \(w \in \Reals_+^V\).
  For every \(y \in \Reals^V\) feasible in the dual
  CQP~\cref{eq:Luz-dual-supp}, we have that \(\oprodsym{y}\) is
  feasible in the dual SDP~\cref{eq:Hdual-dual-SDP}.
  Consequently,
  \begin{equation}
    \label{eq:Hdual-leq-Luz}
    \Hdual(A,w) \leq \Luz(A,w).
  \end{equation}
\end{theorem}
\begin{proof}
  Let \(y \in \Reals^V\) be feasible in the dual
  CQP~\cref{eq:Luz-dual-supp}, so~that
  \(\projsupp w (I + \At)^{\half}y \geq \sqrt{w}\).
  Then \(Y \coloneqq \oprodsym{y}\) is a feasible solution in the dual
  formulation \cref{eq:Hdual-dual-SDP} of \(\Hdual(A,w)\), since
  \(\diag\paren[\big]{(I + \At)^{\half}\oprodsym{y}(I + \At)^{\half}}
  \geq w\).
  Furthermore, the objective values are the same, as
  \(\trace(Y) = \norm[2]{y}^2\).
  Hence, \(\Hdual(A,w) \leq \Luz(A,w)\).
\end{proof}

The relationship between dual feasible solutions for~\(\Hdual\)
and~\(\Luz\) displayed in~\cref{thm:Hdual-leq-Luz} lead to our naming
the bound~\(\Hdual\) as the capital letter for~\(\Luz\).
In~fact, this relationship shows that the new upper bound \(\Hdual\)
can be regarded as semidefinite strengthening of the Luz
bound~\(\Luz\).

The inequality \cref{eq:Hdual-leq-Luz} may be strict, and the gap may
be arbitrarily large.
To see this, let \(G = (V,E)\) be a \(k\)-regular graph, and set
\(A \coloneqq -A_G\).
Note that \(\ones\) is an eigenvector of~\(A\) corresponding to the
eigenvalue \(\lambda_{\min}(A) = -\lambda_{\max}(A_G) = -k\), so that
\((I + \At)\ones = 0\).
If \(w \in \Reals_+^V\) is positive, then \(\projsupp w = I\).
By~\cref{eq:def-Luz-alt-supp},
\begin{equation*}
  \Luz(A, w)
  =
  \sup\setst{
    2\iprodt{\sqrt{w}}{x}
    - \qform{(I+\At)}{x}
  }{
    x \in \Reals_+^V
  }
  \geq
  \sup_{\mu \geq 0} 2\mu\iprodt{\sqrt{w}}{\ones}
  =
  +\infty.
\end{equation*}

We will see next that, when the generalized adjacency matrix
\(A \in \Acal_G\) is required to be nonnegative, equality holds
in~\cref{eq:Hdual-leq-Luz}.
We will use repeatedly that, if \(A \in \Acal_G\) is nonnegative, then
so is~\(\At\).
In this case, for every \(w, x \in \Reals_+^V\), we have that
\[
  \iprodt x {\projsupp w (I + \At) \projsupp w x}
  \le \iprodt x {(I + \At) x}.
\]
Consequently, the constraint ``\(\projsupp{w} x = x\)'' may be dropped
from \cref{eq:def-Luz-alt-supp}: for every \(w \in \Reals_+^V\),
\begin{equation}
  \label{eq:def-Luz-alt}
  \Luz(A, w) 
  = \sup\setst[\big]{
    2\iprodt{\sqrt{w}}{x}
    -
    \qform{(I+\At)}{x}
  }{
    x \in \Reals_+^V
  }
  \qquad
  \text{whenever }
  A \geq 0.
\end{equation}
Accordingly, \cref{eq:Luz-dual-supp} becomes
\begin{equation}
  \label{eq:Luz-dual}
  \Luz(A,w)
  =
  \min\setst[\big]{
    \norm[2]{y}^2
  }{
    y \in \Reals^V,\,
    (I+\At)^{\half} y \geq \sqrt{w}\,
  }
  \qquad
  \text{whenever }
  A \geq 0
  \text{ and }
  \Luz(A,w) < \infty.
\end{equation}

\begin{theorem}
  \label{thm:Luz-pdmg-bound}
  Let \(G = (V,E)\) be a graph.
  For every nonnegative \(A \in \Acal_G\), the function
  \(\Luz_A \colon \Reals_+^V \to \Reals\) is a positive definite
  monotone gauge and, for every \(w \in \Reals_+^V\),
  \begin{equation}
    \label{eq:Luz-pdmg-bound}
    \alpha(G,w) \leq \Luz(A,w) \leq \norm[1]{w}.
  \end{equation}
\end{theorem}
\begin{proof}
  The first inequality of~\cref{eq:Luz-pdmg-bound} is
  just~\cref{eq:alpha-leq-Luz}.
  Let us prove the second inequality in~\cref{eq:Luz-pdmg-bound}.
  In particular, this will show that \(\Luz_A\) is indeed real-valued.
  Let \(w \in \Reals_+^V\).
  For every \(x \in \Reals_+^V\) such that \(\norm[2]{x} = 1\), we have
  \(1 = \iprodt{x}{x} \leq \iprodt{x}{(I + \At)x}\),
  whence
  \begin{multline*}
    \max_{\mu \in \Reals_+}\sqbrac[\big]{
      2 \iprodt{\sqrt{w}}{(\mu x)}
      - \qform{(I + \At)}{(\mu x)}
    }
    \leq
    \max_{\mu \in \Reals_+} \mu\paren{2\iprodt{\sqrt{w}}{x}-\mu}
    =
    \paren[\big]{
      \iprodt{\sqrt{w}}{x}
    }^2
    \leq
    \norm[2]{\sqrt{w}}^2
    \norm[2]{x}^2
    =
    \norm[1]{w}.
  \end{multline*}
  Thus, by~\cref{eq:def-Luz-alt},
  \begin{align*}
    \Luz(A,w)
    &=
    \sup\setst{
      2 \iprodt{\sqrt{w}}{x}
      -
      \qform{(I+\At)}{x}
    }{
      x \in \Reals_+^V
    }
    \\
    &=
    \sup\setst[\bigg]{
      \sup_{\mu \geq 0}\ \sqbrac[\Big]{
        2 \iprodt{\sqrt{w}}{(\mu x)}
        -
        \qform{(I+\At)}{(\mu x)}
      }
    }{
      x \in \Reals_+^V,\,
      \norm[2]{x} = 1
    }
    \leq
    \norm[1]{w}.
  \end{align*}

Next we show that
\begin{equation}
  \label{eq:term-concavity}
  \text{for every \(x \in \Reals_+^V\), the function
    \(w \in \Reals_+^V \mapsto \qform{\Diag(x)(I +
      \At)\Diag(x)}{\sqrt{w}}\) is concave}.
\end{equation}
Let \(x \in \Reals_+^V\), and set
\(Z \coloneqq \Diag(x)(I + \At)\Diag(x)\).
Since the map from~\cref{eq:term-concavity} is clearly positively
homogeneous, it suffices to show that it is superlinear, that~is,
\begin{equation}
  \label{eq:term-concavity-aux0}
  \qform{Z}{\sqrt{w + z}}
  \geq
  \qform{Z}{\sqrt{w}}
  +
  \qform{Z}{\sqrt{z}}
  \qquad
  \forall w,z \in \Reals_+^V.
\end{equation}
Since \(A\) and \(x\) are nonnegative, so is \(Z\).
Hence, since the LHS in~\cref{eq:term-concavity-aux0} is
\(\trace(Z\oprodsym{\sqrt{w+z}})\) and the RHS is
\(\trace(Z\oprodsym{\sqrt{w}}) + \trace(Z\oprodsym{\sqrt{z}})\), it
suffices to prove that
\begin{equation}
  \label{eq:term-concavity-aux1}
  \oprodsym{\sqrt{w+z}}
  \geq
  \oprodsym{\sqrt{w}}
  +
  \oprodsym{\sqrt{z}}
  \qquad
  \forall w,z \in \Reals_+^V.
\end{equation}

Let \(w,z \in \Reals_+^V\) and let \(i,j \in V\).
By the AM--GM inequality,
\begin{align*}
  (w_i + z_i)(w_j + z_j)
  = w_i w_j + z_i z_j + w_i z_j + w_j z_i
  \geq w_i w_j + z_i z_j + 2\sqrt{w_i w_j z_i z_j}
  = \paren[\big]{\sqrt{w_i w_j} + \sqrt{z_i z_j}\,}^2.
\end{align*}
Hence,
\(\sqrt{w_i + z_i}\sqrt{w_j + z_j} \geq \sqrt{w_i w_j} + \sqrt{z_i
  z_j}\), which is the componentwise form
of~\cref{eq:term-concavity-aux1}.
This concludes our proof of~\cref{eq:term-concavity}.

Note that \(\Diag(\sqrt{w}\,)x = \Diag(x)\sqrt{w}\) for every
\(w,x \in \Reals^V\) such that \(w \geq 0\).
Therefore,
\begin{equation*}
  \Luz(A,w)
  =
  \sup\setst[\Big]{
    2\iprodt{w}{x}
    -
    \qform{\Diag(x)(I + \At)\Diag(x)}{\sqrt{w}}
  }{
    x \in \Reals_+
  }.
\end{equation*}
By \cref{eq:term-concavity}, we have just expressed \(\Luz_A\) as the
pointwise supremum of convex functions, whence \(\Luz_A\) is itself
also convex and thus sublinear.
It is also clear that \(\Luz_A\) is positively homogeneous.
Combined with~\cref{eq:alpha-leq-Luz}, we find that \(\Luz_A\) is a
positive definite gauge.

It only remains to prove that \(\Luz_A\) is monotone.
Let \(w,z \in \Reals_+^V\) be such that \(w \leq z\).
Considering the dual formulation of \(\Luz_A\) in
\cref{eq:Luz-dual}, we see that the feasible region for \(z\) is a
subset of the feasible region for~\(w\), since
\(\sqrt{w} \leq \sqrt{z}\).
Since~\cref{eq:Luz-dual} is a minimization problem, we conclude
\(\Luz(A,w) \leq \Luz(A,z)\).
\end{proof}

\begin{theorem}
\label{thm:Luz-eq-Hdual}
Let \(G = (V,E)\) be a graph, and let \(A \in \Acal_G\) be a
nonnegative generalized adjacency matrix of~\(G\).
Then, for every \(w \in \Reals_+^V\),
\[\Luz(A,w) = \Hdual(A,w).\]
\end{theorem}
\begin{proof}
  By \cref{thm:Luz-pdmg-bound}, \(\Luz_A\) is a positive definite
  monotone gauge.
  We first prove that
  \begin{equation}
    \label{eq:Hoffman-eq-Luz-dual}
    \Hoffman_A(z)
    =
    \Luz_A^{\polar}(z)
    \qquad
    \forall z \in \Reals_+^V.
  \end{equation}

  To prove~`\(\geq\)' in~\cref{eq:Hoffman-eq-Luz-dual}, first note
  that \(\Hoffman_A^{\polar}(z) = \Hdual_A(z) \leq \Luz_A(z)\), using
  \cref{thm:Hdual-leq-Luz,thm:Hdual-is-Hoffman-dual}, then apply
  \cref{le:polar-gauge-reversal}, \cref{thm:Hoffman-pdmg-bound}, and
  \cref{item:gauge-duality}.
  For the reverse inequality, let \(z \in \Reals_+^V\), and we shall
  begin by proving that there exists \(y \in \Reals^V\) such that
  \begin{subequations}
    \label{eq:Hoffman-geq-Luz-dual-aux0}
    \begin{gather}
      \label{eq:Hoffman-geq-Luz-dual-aux1}
      \norm[2]{y} = 1,
      \\
      \label{eq:Hoffman-geq-Luz-dual-aux2}
      \Hoffman(A,z)
      =
      \iprodt{y}{
        (I + \At)^{\half}\Diag(z)(I + \At)^{\half}y
      },
      \\
      \label{eq:Hoffman-geq-Luz-dual-aux3}
      (I + \At)^{\half}y \geq 0.
    \end{gather}
  \end{subequations}
  Note that the matrix \((I + \At)\Diag(z)\) is nonnegative.
  Hence, by the Perron-Frobenius Theorem (see
  \cite[Theorem~8.3.1]{HornJ90a}), there exists a nonzero
  \(p \in \Reals_+^V\) such that
  \begin{equation}
    \label{eq:PF-vector}
    (I + \At)\Diag(z)p
    =
    \lambda_{\max}((I + \At)\Diag(z))p.
  \end{equation}
  Set \(\bar{y} \coloneqq (I + \At)^{\MPhalf}p\) and
  \(y \coloneqq \bar{y}/\norm[2]{\bar{y}}\), so that
  \cref{eq:Hoffman-geq-Luz-dual-aux1} holds.
  By \cref{eq:PF-vector}, we have \(p \in \Img(I + \At)\), so that
  \begin{equation*}
    (I+\At)^{\half}\bar{y} = (I + \At)^{\half}
    (I + \At)^{\MPhalf}p = p \geq 0,
  \end{equation*}
  so~\cref{eq:Hoffman-geq-Luz-dual-aux3} holds.
  Furthermore, \cref{lem:eigen-swap,cor:Hoffman-swapped} imply that
  \[
    \lambda_{\max}((I + \At)\Diag(z))
    = \lambda_{\max}\paren*{(I + \At)^{\half}\Diag(z)(I + \At)^{\half}}
    = \Hoffman(A,z).
  \]
  Hence, by applying \((I + \At)^{\MPhalf}\) to both sides of
  \cref{eq:PF-vector} we conclude that
  \[
    \paren[\big]{
      (I + \At)^{\half}\Diag(z)(I + \At)^{\half}
    }
    \bar{y}
    =
    \Hoffman(A,z) \bar{y},
  \]
  whence \cref{eq:Hoffman-geq-Luz-dual-aux2} follows.

  Recall that by \cref{eq:def-polar-gauge}, we have that
  \(\Luz_A^{\polar}(z) = \max_{w \in \Ccal} \iprodt{w}{z}\), where
  \begin{equation*}
    \Ccal \coloneqq \setst{w \in \Reals_+^V}{\Luz_A(w) \leq 1}.
  \end{equation*}
  Define
  \(
  w
  \coloneqq
  \diag\paren[\big]{
    (I + \At)^{\half}\oprodsym{y}(I + \At)^{\half}
  }.
  \)
  Note that~\(y\) is a feasible solution for the
  dual~\cref{eq:Luz-dual} of~\(\Luz(A,w)\)
  by~\cref{eq:Hoffman-geq-Luz-dual-aux3}, with objective value
  \(\norm[2]{y}^2 = 1\) by~\cref{eq:Hoffman-geq-Luz-dual-aux1}.
  Hence, \(w \in \Ccal\), and by~\cref{eq:Hoffman-geq-Luz-dual-aux2}
  we get
  \begin{equation*}
    \Hoffman(A,z)
    =
    \iprodt{y}{(I + \At)^{\half}\Diag(z)(I + \At)^{\half}y}
    =
    \iprodt{w}{z}
    \leq
    \Luz(A,z)^{\polar}.
  \end{equation*}
  This concludes the proof of~`\(\leq\)'
  in~\cref{eq:Hoffman-eq-Luz-dual}, and hence that
  of~\cref{eq:Hoffman-eq-Luz-dual} itself.

  By applying the gauge dual to both sides
  of~\cref{eq:Hoffman-eq-Luz-dual}, and using
  \cref{thm:Hdual-is-Hoffman-dual,%
    thm:Luz-pdmg-bound,item:gauge-duality}, we find that
  \(\Hdual_A = \Hoffman_A^{\polar} = \Luz_A^{\polar\polar} = \Luz_A\).
\end{proof}

As a consequence, the dual SDP~\cref{eq:Hdual-dual-SDP} always has an
optimal solution that is rank-one if \(A\) is a nonnegative
generalized adjacency matrix:
\begin{corollary}
  Let \(G = (V,E)\) be a graph, and let \(A \in \Acal_G\) be a
  nonnegative generalized adjacency matrix of~\(G\).
  Then for every \(w \in \Reals_+^V\), the dual
  SDP~\cref{eq:Hdual-dual-SDP} has a rank-one optimal solution.
\end{corollary}
\begin{proof}
  Let \(y \in \Reals^V\) be an optimal solution for the dual
  CQP~\cref{eq:Luz-dual}.
  \cref{thm:Hdual-leq-Luz} states that \(yy^\transp\) is feasible in
  \cref{eq:Hdual-dual-SDP}, and \cref{thm:Luz-eq-Hdual} implies it is
  optimal, since
  \(\trace(\oprodsym{y}) = \norm[2]{y}^2 = \Luz(A,w)
  = \Hdual(A,w)\).
\end{proof}

\section{Relation with The Lovász Theta Function and Its Variants}
\label{sec:theta}

Let \(G = (V,E)\) be a graph.
The \emph{theta body} of~\(G\) is defined as the projection
\begin{equation}
  \label{eq:def-TH}
  \THbody(G)
  \coloneqq
  \setst*{
    x \in \Reals_+^V
  }{
    \exists X \in \Psd{V},\,
    \diag(X) = x,\,
    \begin{bmatrix}
      1        & x^{\transp}\, \\
      x & X
    \end{bmatrix}
    \succeq 0,\,
    X_{ij} = 0,\,\forall ij \in E
  }.
\end{equation}
Since \(\THbody(G)\) is a linear projection of a compact convex set,
\(\THbody(G)\) is compact and convex.
It is not hard to verify that \(\THbody(G) \subseteq [0,1]^V\) is
lower-comprehensive and that \(\frac{1}{n}\ones \in \THbody(G)\),
which implies that \(\frac{1}{2n}\ones\) is in the interior of
\(\THbody(G)\).
Thus, \(\THbody(G)\) is a convex corner.
Hence, using \cref{item:gauge-construction-dual}, we may finally
define the following positive definite monotone gauge:
\begin{equation*}
  \theta(G,w)
  \coloneqq
  \max\setst[\big]{
    \iprodt{w}{x}
  }{
    x \in \THbody(G)
  },
  \qquad
  \forall w \in \Reals_+^V.
\end{equation*}
The reader is referred to~\cite[Theorem~14]{CarliT17a} for the
equivalence with other definitions of~\(\theta(G,w)\).
Alternative sources for~\(\theta\) include
\cite{Lovasz79a,GroetschelLS86a,GroetschelLS93a,Knuth94a}.
It is known (see~\cite[Corollary~3.4]{GroetschelLS86a}) that
\begin{equation}
  \label{eq:abl-TH}
  \abl(\THbody(G)) = \THbody(\overline{G}),
\end{equation}
so that by \cref{item:gauge-construction-dual} we have
\begin{equation}
  \label{eq:dual-of-theta}
  \theta_G^{\polar} = \theta_{\,\overline{G}}.
\end{equation}

It was already mentioned in \cref{eq:theta-best-Luz-pre} that
\begin{equation}
  \label{eq:theta-best-Luz}
  \theta(G,w) = \min_{A \in \Acal_G} \Luz(A,w),
\end{equation}
and \(\theta(G,w)\) is similarly related
(see~\cite[Theorem~33]{Knuth94a}) to the Hoffman bounds~\(\Hoffman_A\):
\begin{equation}
  \label{eq:theta-bar-Hoffman}
  \theta(\overline{G},w) = \max_{A \in \Acal_G} \Hoffman(A,w).
\end{equation}
We remark that \cite[Theorem~33]{Knuth94a} essentially proves that,
for \(W \coloneqq \Diag(w)\), we have
\begin{equation*}
  \theta(\overline{G},w)
  =
  \max \setst[\Big]{
    \lambda_{\max}\paren[\big]{W^{\half}(I + \mu A)W^{\half}}
  }{
    A \in \Acal_G,\,
    0 \leq \mu \leq \tfrac{1}{-\lambda_{\min}(A)}
  },
\end{equation*}
from which one can get \cref{eq:theta-bar-Hoffman} by using convexity
of the function
\(\mu \in \Reals \mapsto \lambda_{\max}(W^{\half}(I + \mu
A)W^{\half})\).

The use of `\(\min\)' in~\cref{eq:theta-best-Luz} means that there
exists \(A \in \Acal_G\) such that \(\theta(G,w) = \Luz(A,w)\), and
analogously for \cref{eq:theta-bar-Hoffman}.
The same applies to the next result involving~\(\Hdual\):
\begin{theorem}
  \label{thm:theta-best-Upsilon}
  Let \(G = (V,E)\) be a graph.
  Then, for every \(w \in \Reals_+^V\),
  \begin{equation}
    \label{eq:theta-best-Upsilon}
    \theta(G,w)
    =
    \min_{A \in \Acal_G} \Hdual(A,w).
  \end{equation}
\end{theorem}
\begin{proof}
  \Cref{thm:Hdual-leq-Luz,thm:Hdual-is-Hoffman-dual},
  \cref{eq:theta-bar-Hoffman,le:polar-gauge-reversal}, and
  \cref{eq:dual-of-theta} imply that
  \(
    \Luz_A
    \geq
    \Hdual_A
    =
    \Hoffman_A^{\polar}
    \geq
    \theta_{\,\overline{G}}^{\polar}
    =
    \theta_G
  \)
  for every \(A \in \Acal_G\).
  \Cref{eq:theta-best-Luz} then finishes the proof, while also proving
  that there exists \(A \in \Acal_G\) such that equality holds.
\end{proof}

\Cref{eq:theta-bar-Hoffman,eq:theta-best-Upsilon}, together with
\cref{thm:Hoffman-pdmg-bound,cor:Hdual-pdmg-bound}, show that
\(\theta\) may be expressed using optimization over positive definite
monotone gauges, each of which can be expressed as linear optimization
over some convex corners, since
\begin{gather*}
  \Hdual(A,w) = \max_{x \in \Ucal_A} \iprodt{w}{x},
  \qquad
  \forall w \in \Reals_+^V,
  \\
  \Hoffman(A,w) = \max_{x \in \Hcal_A} \iprodt{w}{x},
  \qquad
  \forall w \in \Reals_+^V.
\end{gather*}
Moreover, by \cref{thm:Hoffman-max-convex},
\begin{equation}
  \label{eq:abl-U}
  \abl(\Ucal_A)
  =
  \setst{
    x \in \Reals_+^V
  }{
    \iprodt{w}{x} \leq 1,\,
    \forall w \in \Ucal_A
  }
  =
  \setst{
    x \in \Reals_+^V
  }{
    \Hdual_A(x) \leq 1
  }
  =
  \Hcal_A.
\end{equation}
These provide alternative descriptions of \(\THbody(G)\) via such
convex corners.
We shall rely on the following fact, for every family \(\Fcal\) of
subsets of~\(\Reals_+^V\):
\begin{equation}
  \begin{split}
    \label{eq:abl-of-union}
    \abl\paren[\Big]{\bigcup \Fcal}
    &=
    \setst[\big]{
      x \in \Reals_+^V
    }{
      \textstyle
      \iprodt{y}{x} \leq 1,\,
      \forall y \in \bigcup \Fcal
    }
    \\
    &=
    \setst[\big]{
      x \in \Reals_+^V
    }{
      \iprodt{y}{x} \leq 1,\,
      \forall S \in \Fcal,\,
      \forall y \in S
    }
    \\
    &=
    \setst[\big]{
      x \in \Reals_+^V
    }{
      x \in \abl(S),\,
      \forall S \in \Fcal
    }\\
    &=
    \bigcap_{S \in \Fcal} \abl(S).
  \end{split}
\end{equation}
\begin{proposition}
  \label{prop:theta-body-description}
  Let \(G = (V,E)\) be a graph.
  Then
  \begin{align}
    \label{eq:theta-body-description}
    \THbody(G) &= \bigcap_{A \in \Acal_G} \Ucal_A,\\
    \label{eq:theta-bar-body-description}
    \THbody(\overline{G}) &= \bigcup_{A \in \Acal_G} \Hcal_A.
  \end{align}
\end{proposition}
\begin{proof}
  Let \(w \in \Reals_+^V\).
  \Cref{thm:theta-best-Upsilon} implies that \(\theta(G,w) \leq 1\)
  holds if and only if there exists \(A \in \Acal_G\) such that
  \(\Hdual(A,w) \leq 1\).
  Since
  \(\THbody(\overline{G}) = \setst{w \in \Reals_+^V}{\theta(G,w) \leq
    1}\) by \cref{eq:abl-TH} and
  \(\Hcal_A = \setst{w \in \Reals_+^V}{\Hdual(A,w) \leq 1}\), we have
  \cref{eq:theta-bar-body-description}.
  Then \cref{eq:abl-TH,eq:abl-of-union,eq:abl-U,eq:abl-abl} finish the
  proof:
  \[
    \THbody(G)
    = \abl(\THbody(\overline{G}))
    = \abl\paren[\Bigg]{\bigcup_{A \in \Acal_G} \Hcal_A}
    = \bigcap_{A \in \Acal_G} \abl(\Hcal_A)
    = \bigcap_{A \in \Acal_G} \Ucal_A.
    \qedhere
  \]
\end{proof}
\Cref{eq:theta-bar-body-description} is slightly unusual since it is
not \emph{a~priori} clear that the union of convex sets in the RHS is
convex.

Next we discuss the variants \(\theta'\) and~\(\theta^+\) of a graph
\(G = (V,E)\).
We first define \(\THbody^+(G)\) by relaxing the constraint
``\(X_{ij} = 0\)'' in \cref{eq:def-TH} to ``\(X_{ij} \leq 0\)'', for
each edge \(ij \in E\).
Next we define \(\THbody'(G)\) by adding to \cref{eq:def-TH} the
constraint \(X_{ij} \geq 0\) for each \(i,j \in V\).
The two resulting sets can be verified to be convex corners,
analogously to~\(\THbody(G)\).
Then the variants \(\theta'\) and \(\theta^+\), which are also
positive definite monotone gauges by
\cref{item:gauge-construction-dual}, are defined as
\begin{gather}
  \theta'(G,w)
  \coloneqq
  \max\setst[\big]{
    \iprodt{w}{x}
  }{
    {x \in \THbody'(G)}
  }
  \qquad
  \forall w \in \Reals_+^V,
  \\
  \theta^+(G,w)
  \coloneqq
  \max\setst{
    \iprodt{w}{x}
  }{
    x \in \THbody^+(G)
  }
  \qquad
  \forall w \in \Reals_+^V.
\end{gather}
It is well known that
\begin{equation}
  \label{eq:abl-TH'}
  \abl\paren[\big]{\THbody'(G)}
  =
  \THbody^+(\overline{G}),
\end{equation}
so that by \cref{item:gauge-construction-dual} we have
\begin{equation}
  \label{eq:dual-of-theta'}
  \theta'(G,\cdot)^{\polar}
  =
  \theta^+(\overline{G},\cdot).
\end{equation}

The previous relations \cref{eq:theta-best-Luz,eq:theta-bar-Hoffman}
may be adapted to \(\theta'\) and \(\theta^+\) by restricting the
range of the maxima/minima over \emph{nonnegative} generalized
adjacency matrices.
For a graph \(G\), denote
\begin{equation}
  \label{eq:def-A-nonneg}
  \Acal_G^+
  \coloneqq
  \setst{A \in \Acal_G}{A \geq 0}.
\end{equation}
Then \cite[Corollary~29]{CarliT17a} shows that
\begin{equation}
  \label{eq:theta+-best-Luz}
  \theta^+(G,w) = \min_{A \in \Acal_G^+} \Luz(A,w),
\end{equation}
and it is well known (see \cite[Proposition~21]{CarliT17a} and the
remark following \cref{eq:theta-bar-Hoffman}) that
\begin{equation}
  \label{eq:theta'-bar-Hoffman}
  \theta'(\overline{G},w) = \max_{A \in \Acal_G^+} \Hoffman(A,w).
\end{equation}

Let us now state a counterpart to \cref{thm:theta-best-Upsilon}:
\begin{theorem}
  \label{thm:theta+-best-Upsilon}
  Let \(G = (V,E)\) be a graph.
  Then, for every \(w \in \Reals_+^V\),
  \begin{equation}
    \label{eq:theta+-best-Upsilon}
    \theta^+(G,w)
    =
    \min_{A \in \Acal_G^+} \Hdual(A,w).
  \end{equation}
\end{theorem}
\begin{proof}
  Immediate from \cref{eq:theta+-best-Luz,thm:Luz-eq-Hdual}.
\end{proof}

And our final result is a counterpart to
\cref{prop:theta-body-description}:
\begin{proposition}
Let \(G = (V,E)\) be a graph.
Then
\begin{align}
  \label{eq:theta-plus-body-description}
  \THbody^+(G) &= \bigcap_{A \in \Acal_G^+} \Ucal_A,\\
  \label{eq:theta-prime-body-description}
  \THbody'(\overline{G}) &= \bigcup_{A \in \Acal_G^+} \Hcal_A.
\end{align}
\end{proposition}
\begin{proof}
  Let \(w \in \Reals_+^V\).
  \Cref{thm:theta+-best-Upsilon} implies that \(\theta^+(G,w) \leq 1\)
  holds if and only if there exists \(A \in \Acal_G^+\) such that
  \(\Hdual(A,w) \leq 1\).
  \Cref{eq:abl-TH',eq:abl-abl} imply
  \(\THbody'(\overline{G}) = \setst{w \in \Reals_+^V}{\theta^+(G,w)
    \le 1}\).
  Moreover, as
  \(\Hcal_A = \setst{w \in \Reals_+^V}{\Hdual(A,w) \leq 1}\), we have
  \cref{eq:theta-prime-body-description}.
  Then \cref{eq:abl-TH',eq:abl-of-union,eq:abl-U,eq:abl-abl} finish the
  proof:
  \[
    \THbody^+(G)
    = \abl(\THbody'(\overline{G}))
    = \abl\paren[\Bigg]{\bigcup_{A \in \Acal_G^+} \Hcal_A}
    = \bigcap_{A \in \Acal_G^+} \abl(\Hcal_A)
    = \bigcap_{A \in \Acal_G^+} \Ucal_A.
    \qedhere
  \]
\end{proof}

\bibliography{main}
\bibliographystyle{plain}

\end{document}